\documentclass[11pt,reqno,a4paper]{amsart}

\usepackage{hyperref}
\usepackage{latexsym}
\usepackage{amssymb,amsmath}
\usepackage{graphics}
\usepackage[usenames,dvipsnames]{color}
\usepackage{url}
\usepackage{bbm}
\usepackage[vcentermath]{youngtab}
\usepackage{booktabs}
\usepackage{enumerate}
\usepackage[shortlabels]{enumitem}
\usepackage{lastpage}
\usepackage{extarrows}
\usepackage{graphicx}
\usepackage{outlines}
\usepackage[arrow,matrix,curve,cmtip,ps]{xy}
\usepackage{varwidth}
\usepackage{amsthm}
\usepackage[font=itshape]{quoting} 
\usepackage{float}
\usepackage{array}
\usepackage{mathrsfs} 
\usepackage{bm} 
\usepackage{arydshln} 
\allowdisplaybreaks 
\usepackage[font=small]{caption,subcaption}
\usepackage{multicol}

\usepackage{tikz}
\usetikzlibrary{matrix,chains,shapes,arrows,scopes,positioning}
\usetikzlibrary{calc,decorations.pathmorphing,shapes}
\usetikzlibrary{patterns}
\tikzstyle{empty}=[circle,draw=black!80,thick]
\tikzstyle{emptyn}=[circle,draw=black!80,fill=white,scale=0.5] 
\tikzstyle{nero}=[circle,draw=black!80,fill=black!80,thick]

\newcommand{\6}{^}
\newcommand{\Syl}{\operatorname{Syl}}
\newcommand{\Lin}{\operatorname{Lin}}
\newcommand{\Irr}{{\operatorname{Irr}}}

\newcommand{\Char}{{\operatorname{Char}}}
\newcommand{\fS}{\mathfrak{S}}

\newcommand{\triv}{\mathbbm{1}}
\newcommand{\down}{\big\downarrow}
\newcommand{\up}{\big\uparrow}
\newcommand{\tworow}[2]{\mathsf{t}_{#1}[{#2}]}
\newcommand{\hook}[2]{\mathsf{h}_{#1}[{#2}]}

\newtheorem{thm}{Theorem}[section]
\newtheorem{lemma}[thm]{Lemma}

\newtheorem{cor}[thm]{Corollary}
\newtheorem{prop}[thm]{Proposition}

\newtheorem*{thmA}{Theorem A}
\newtheorem*{thmA1}{Main Result A}
\newtheorem*{thmB}{Theorem B}
\newtheorem*{thmB1}{Main Result B}
\newtheorem*{thmC}{Theorem C}

\theoremstyle{definition}
\newtheorem{eg}[thm]{Example}
\newtheorem{defn}[thm]{Definition}

\newtheorem{rem}[thm]{Remark}

\begin{document}

\title{Sylow Branching Coefficients for symmetric groups}


\author{Eugenio Giannelli}
\address[E. Giannelli]{Dipartimento di Matematica e Informatica	U.~Dini, Viale Morgagni 67/a, Firenze, Italy}
\email{eugenio.giannelli@unifi.it}

\author{Stacey Law}
\address[S. Law]{Mathematical Institute, University of Oxford, Radcliffe Observatory, Andrew Wiles Building, Woodstock Rd, Oxford OX2 6GG, UK}
\email{swcl2@cam.ac.uk}

\begin{abstract}
Let $p\geq 5$ be a prime and let $n$ be a natural number. In this article we describe the irreducible constituents of the induced characters $\phi\up^{\fS_n}$ for arbitrary linear characters $\phi$ of a Sylow $p$-subgroup $P_n$ of the symmetric group $\fS_n$, generalising results of \cite{GL1}. By doing this, we introduce \textit{Sylow branching coefficients} for symmetric groups. 
\end{abstract}

\keywords{}


\maketitle

\section{Introduction}\label{sec:intro}
The study of the relationship between the representation theory of a finite group and that of its Sylow subgroups has been a central topic of research in the last few decades \cite{NavSylSurvey}.
For instance, Problem 12 of Brauer's article \cite{Brauer} and the famous Brauer Height Zero Conjecture \cite[Problem 23]{Brauer} ask what amount of the algebraic structure of a Sylow $p$-subgroup can be read off the character table of a finite group. 
More recently, it has been noted that given a finite group $G$ with Sylow $p$-subgroup $P$, the permutation character $\triv_P\up\6G$ controls important structural properties of the entire group $G$. For example, in \cite{MN} it is shown that $P$ is normal in $G$ if and only if all irreducible constituents of $\triv_P\up\6G$ have degree coprime to $p$. At the opposite end of the spectrum, in \cite{NTV} it is shown that when $p$ is odd then the Sylow $p$-subgroup $P$ is self-normalising if and only if $\triv_G$ is the only constituent of $\triv_P\up\6G$ of degree coprime to $p$.
 
\smallskip 
 
The abundant and deep knowledge on the representation theory of symmetric groups often allows one to ask (and sometimes answer) questions about this family of finite groups which are out of reach for arbitrary groups. 
This is the case for the study of the interplay between characters of $\mathfrak{S}_n$ and those of its Sylow $p$-subgroup $P_n$. 
Our main object of investigation is restriction of irreducible characters of $\mathfrak{S}_n$ to $P_n$ and their decomposition into irreducible constituents. 
In particular, for $\chi\in\mathrm{Irr}(\fS_n)$ we let 
$$\chi\down_{P_n}=\sum_{\phi\in\mathrm{Irr}(P_n)}Z_{\phi}^\chi\phi,$$
where each \textit{Sylow branching coefficient} $Z_{\phi}^\chi\in\mathbb{N}_0$ is the
multiplicity of $\phi$ as an irreducible constituent of $\chi\down_{P_n}$. 
Letting $\triv_{P_n}$ denote the trivial character of $P_n$, the positivity of $Z^\chi_{\triv_{P_n}}$ was completely described in \cite{GL1}, for odd primes. 
In this article we largely extend the work of \cite{GL1} by considering the entire set $\mathrm{Lin}(P_n)$ of linear characters of $P_n$. In particular, for any linear character $\phi$ of $P_n$ we study  the set $\Omega(\phi)$ consisting of all those irreducible characters $\chi$ of $\mathfrak{S}_n$ such that $Z^\chi_\phi\neq 0.$

Fix a prime $p\geq 5$, let $n\in\mathbb{N}$ and $P_n\in\Syl_p(\fS_n)$. Let $\phi$ be any linear character of $P_n$. 
We recall that the set $\Irr(\fS_n)$ of ordinary irreducible characters of $\fS_n$ is naturally in bijection with the set $\mathcal{P}(n)$ of partitions of $n$, and for any $\lambda\in\mathcal{P}(n)$ we let $\chi^\lambda\in\mathrm{Irr}(\fS_n)$ be its corresponding irreducible character.
Thus we may view $\Omega(\phi)$ as a subset of $\mathcal{P}(n)$; in other words, we set
$$\Omega(\phi) = \{\lambda\in\mathcal{P}(n)\ |\ Z^{\lambda}_{\phi}\neq 0 \},$$
where, for simplicity, we used the symbol $Z^{\lambda}_{\phi}$ to denote $Z^{\chi^\lambda}_{\phi}$.

Our first result gives a complete description of $\Omega(\phi)$ for a certain family of linear characters of $P_n$ which we call \textit{quasi-trivial} linear characters (see Definition~\ref{def: quasi-trivial}). This is a broad extension of \cite[Theorem A]{GL1}. The precise statement depends on the identification of linear characters of $P_n$ with multisets of sequences, as explained in full detail in Section~\ref{sec: sylow prelims}.
For this reason we give just an informal description of our result here and postpone the full statement of Theorem A to Section~\ref{sec: sylow prelims}.

\begin{thmA1}
Let $p\geq 5$ be a prime and let $n\in\mathbb{N}$. Let $\phi\in\mathrm{Lin}(P_n)$ be quasi-trivial. Then we completely describe the set $\Omega(\phi)$ by finding explicitly all partitions contained in it.
\end{thmA1}

We repeat that the reader is referred to Theorem A in Section~\ref{sec: sylow prelims} below for the precise statement. What is important for now is that this result gives an exact characterisation of $\Omega(\phi)$ for all quasi-trivial $\phi\in\mathrm{Lin}(P_n)$. In order to describe the sets $\Omega(\phi)$ for all linear characters $\phi$ of $P_n$, we first define for any $n,t\in\mathbb{N}$ the set 
$$\mathcal{B}_n(t):=\{\lambda\in\mathcal{P}(n) : \lambda_1\le t,\ l(\lambda)\le t \}.$$ 
Here $\lambda_1$ and $l(\lambda)$ denote the length of the first row and first column of $\lambda$ respectively.
In particular, $\mathcal{B}_n(t)$ is the set of partitions of $n$ whose Young diagrams fit inside a $t\times t$ square grid.
Finally, we let $m(\phi)$ and $M(\phi)$ be the integers defined as follows:
$$m(\phi) := \max\{t\in\mathbb{N} \mid \mathcal{B}_{n}(t) \subseteq\Omega(\phi)\}\quad \text{and}\quad  M(\phi) := \min\{t\in\mathbb{N} \mid \Omega(\phi)\subseteq\mathcal{B}_{n}(t)\}.$$

The main result of this article is Theorem B, which is stated in Section~\ref{sec: sylow prelims}. Here we avoid the necessary technical notation and limit ourselves to describing it informally. 

\begin{thmB1}
Let $p\geq 5$ be a prime and $n$ a natural number. Let $P_n\in\Syl_p(\fS_n)$ and let $\phi$ be any linear character of $P_n$. Then the values of $m(\phi)$ and $M(\phi)$ are explicitly computed.
\end{thmB1}

Theorem B translates into a very precise description of $\Omega(\phi)$ for all $\phi\in\mathrm{Lin}(P_n)$. In fact $$\mathcal{B}_n(m(\phi))\subseteq\Omega(\phi)\subseteq\mathcal{B}_n(M(\phi)),$$ and we will show that that the values $M(\phi)$ and $m(\phi)$ are close to each other for all $\phi\in\mathrm{Lin}(P_n)$. 

We conclude with an asymptotic result, which is somewhat curious. Namely, \textit{almost all} $\chi\in\mathrm{Irr}(\mathfrak{S}_n)$ share the following property: $Z_\phi^{\chi}\neq 0$ for all $\phi\in\mathrm{Lin}(P_n)$. This is  expressed precisely in the following theorem.

\begin{thmC}
Let $p\geq 5$ be a prime and $n\in\mathbb{N}$. Let $\Omega_n$ be the intersection of all the sets $\Omega(\phi)$ where $\phi$ is free to run among the elements of $\Lin(P_n)$. Then 
$$\lim_{n\rightarrow\infty}\frac{|\Omega_n|}{|\mathcal{P}(n)|}=1.$$
\end{thmC}

\medskip

\begin{rem}
This article studies Sylow branching coefficients for primes $p\geq 5$. 
We take the opportunity to discuss the main differences and the obstacles that arise when studying this problem for the primes $2$ and $3$. 

When $p=2$ the restriction of irreducible characters to Sylow $2$-subgroups was studied for its connections to the McKay Conjecture in \cite{G}, \cite{GKNT} and \cite{INOT}.
In this setting the situation is completely different from the uniform description given in this article for all primes $p\geq 5$. 
For instance, $\Omega(\phi)$ is no longer closed under conjugation in general, and at the time of writing we do not even have a conjecture for the structure of the sets $\Omega(\phi)$, where $\phi$ is a linear character of a Sylow $2$-subgroup of $\mathfrak{S}_n$. Indeed, a first open problem in this line of investigation for the prime $2$ is to determine $\Omega(\triv_{P_n})$. 
A second question is whether Theorem C would still hold for the prime $2$. 

For the prime $3$ the set $\Omega(\triv_{P_n})$ was described in \cite{GL1}.
However, Theorems A and B of the present article (see Section \ref{sec: sylow prelims}) do not hold for $p=3$. We refer the reader to Example \ref{counterexample} for concrete cases of linear characters $\phi$ of Sylow $3$-subgroups such that the structure of $\Omega(\phi)$ does not agree with the one described by Theorems A and B. 
Despite the many differences in the behaviour of Sylow branching coefficients for the prime $3$ compared to bigger primes, the ideas contained in this article, together with new and more sophisticated ad hoc combinatorial machinery, should allow us to address the problem and to obtain results similar to Main Result A and Main Result B for the prime $3$ as well. 
In particular, we conjecture that Theorem C holds for $p=3$.
This specific and very technical analysis will be the subject of future investigation. 
\hfill$\lozenge$
\end{rem}

\medskip

The article is structured as follows. In Section~\ref{sec:prelims} we recall basic facts in the representation theory of symmetric groups and their Sylow $p$-subgroups. This allows us to formally state our main results (Theorems A and B). 
In Section~\ref{sec:LR} we set up the combinatorial background necessary to tackle the main proofs. 
In Section~\ref{sec: new omegas prime power}, we consider the case where $n$ is a power of the prime $p$, and in Section~\ref{sec: new omegas arbitrary n} we extend the scope of our results to arbitrary natural numbers $n$. 

\medskip

\subsection*{Acknowledgements}
The second author was supported by an LMS Early Career Fellowship at the University of Oxford.
We are indebted to Jason Long for his help and for useful conversations on Section 3. We also thank Mark Wildon for helpful discussions on \cite{DPW} and  \cite{PW}.

\medskip

\section{Notation, Preliminaries, and statements of Theorems A and B}\label{sec:prelims}
Throughout this article, $p$ is a prime and $P_n$ denotes a Sylow $p$-subgroup of $\fS_n$. 
For a finite group $G$, let $\Char(G)$ denote the set of ordinary characters of $G$, and let $\Irr(G)$ (resp.~$\Lin(G)$) denote the subset of those which are irreducible (resp.~linear).
For $m$ a natural number, let $[m]$ denote the set $\{1,2,\dotsc,m\}$ and $[\overline{m}]$ the set $\{0,1,\dotsc,m-1\}$. 

\subsection{Wreath products}\label{sec:wreath}
We use this section to fix the notation for representations and characters of wreath products. This will be important for our study of Sylow $p$-subgroups of $\mathfrak{S}_n$.

Let $G$ be a finite group, $n\in\mathbb{N}$ and $H\le\fS_n$. We denote by $G^{\times n}$ the direct product of $n$ copies of $G$. The natural action of $\fS_n$ on the direct factors of $G^{\times n}$ induces an action of $\fS_n$ (and therefore of $H\le \fS_n$) via automorphisms of $G^{\times n}$, giving the wreath product $G\wr H:= G^{\times n}\rtimes H$. We sometimes refer to $G^{\times n}$ as the base group of the wreath product $G\wr H$. 

As in \cite[Chapter 4]{JK}, we denote the elements of $G\wr H$ by $(g_1,\dotsc,g_n;h)$ for $g_i\in G$ and $h\in H$. Let $V$ be a $\mathbb{C}G$--module and suppose it affords the character $\phi$. 
We let $V^{\otimes n}:=V\otimes\cdots\otimes V$ ($n$ copies) be the corresponding $\mathbb{C}G^{\times n}$--module. The left action of $G\wr H$ on $V^{\otimes n}$ defined by linearly extending
$$(g_1,\dotsc,g_n;h)\ :\quad v_1\otimes \cdots\otimes v_n \longmapsto g_1v_{h^{-1}(1)}\otimes\cdots\otimes g_nv_{h^{-1}(n)}$$
turns $V^{\otimes n}$ into a $\mathbb{C}(G\wr H)$--module, which we denote by $\widetilde{V^{\otimes n}}$ (see \cite[(4.3.7)]{JK}). We denote by $\tilde{\phi}$ the character afforded by the $\mathbb{C}(G\wr H)$--module $\widetilde{V^{\otimes n}}$. For any character $\psi$ of $H$, we let $\psi$ also denote its inflation to $G\wr H$ and let
$$\mathcal{X}(\phi;\psi):=\tilde{\phi}\cdot\psi$$
be the  character of $G\wr H$ obtained as the
product of $\tilde{\phi}$ and $\psi$.
Moreover, if $K\le G$ and $L\le H$ then we have by the definition of $\mathcal{X}(\phi;\psi)$ that
$$ \mathcal{X}(\phi;\psi)\down^{G\wr H}_{K\wr L} = \mathcal{X}(\phi\down^G_K; \psi\down^H_L).$$

\smallskip

Let $\phi\in\Irr(G)$ and let $\phi^{\times n}:=\phi\times\cdots\times\phi$ denote the corresponding irreducible character of $G^{\times n}$, and observe that $\tilde{\phi}\in\Irr(G\wr H)$ is an extension of $\phi^{\times n}$. For $\psi\in\Irr(H)$ we have that $\mathcal{X}(\phi;\psi)\in\Irr(G\wr H\mid \phi^{\times n})$, the set of irreducible characters $\chi$ of $G\wr H$ whose restriction $\chi\down_{G^{\times n}}$ contains $\phi^{\times n}$ as an irreducible constituent. Indeed, Gallagher's Theorem \cite[Corollary 6.17]{IBook} gives
$$\Irr(G\wr H\mid \phi^{\times n}) = \{\mathcal{X}(\phi;\psi)\mid \psi\in\Irr(H)\}.$$
More generally, if $K\le G$ and $\psi\in\Irr(K)$ then we denote by $\Irr(G\mid\psi)$ the set of characters $\chi\in\Irr(G)$ such that $\psi$ is an irreducible constituent of the restriction $\chi\down_K$.

\smallskip

We also record the form of irreducible characters of $G\wr C_p$ where $C_p$ is a cyclic group of prime order $p$ (see \cite[Chapter 4]{JK}): every $\psi\in\Irr(G\wr C_p)$ is either of the form
\begin{itemize}\setlength\itemsep{0.5em}
\item[(a)] $\psi=\phi_{i_1}\times\cdots\times\phi_{i_p}\up^{G\wr C_p}_{G^{\times p}}$, where $\phi_{i_1},\dotsc,\phi_{i_p}\in\Irr(G)$ are not all equal; or
\item[(b)] $\psi=\mathcal{X}(\phi;\theta)$ for some $\phi\in\Irr(G)$ and $\theta\in\Irr(C_p)$.
\end{itemize}

When (a) holds, $\psi\down_{G^{\times p}}$ is the sum of the $p$ irreducible characters of $G^{\times p}$ whose $p$ factors are a cyclic permutation of $\phi_{i_1},\dotsc,\phi_{i_p}$. When (b) holds, $\psi\down_{G^{\times p}} = \phi^{\times p}\cdot\theta(1) = \phi^{\times p}$.

\begin{lemma}[Associativity of wreath products]\label{lem:assoc}
	Let $l,m,n\in\mathbb{N}$ and let $G\le\fS_l$, $H\le\fS_m$ and $I\le\fS_n$. Then $(G\wr H)\wr I \cong G\wr (H\wr I)$. Moreover, given $\alpha\in\Char(G)$, $\beta\in\Char(H)$, $\gamma\in\Char(I)$ and $x\in (G\wr H)\wr I$, we have that 
	$\mathcal{X}\big(\mathcal{X}(\alpha;\beta);\gamma \big)(x) = \mathcal{X}\big(\alpha; \mathcal{X}(\beta;\gamma) \big)(\theta(x)).$ Here $\theta$ denotes the canonical isomorphism between $(G\wr H)\wr I$ and $G\wr (H\wr I)$.
\end{lemma}

\begin{proof}
	The first statement is a routine check, following the notational convention in \cite[\textsection 4.1]{JK}. The second statement follows from the formula for character values of wreath product given in \cite[Lemma 4.3.9]{JK}.
\end{proof}

In particular, associativity for three terms as in Lemma~\ref{lem:assoc} then gives associativity for $k$-term wreath products for all $k\ge 3$, and so from now on we simply write $G_1\wr G_2\wr\cdots\wr G_k$ without internal parentheses when referring to such groups, and identify corresponding elements under such isomorphisms.

\smallskip

We record some useful results describing the irreducible constituents of restrictions and inductions of characters of wreath products.
The first is entirely elementary: we state it here for the reader's convenience as it will be used later in the article. 

\begin{lemma}\label{lem: easy observation}
Let $G$ be a finite group and $H\le\fS_n$ for some $n\in\mathbb{N}$. Let $\chi\in\Irr(G)$. Then $$\chi^{\times n}\up^{G\wr H}_{G^{\times n}} = \sum_{\theta\in\Irr(H)} \theta(1)\cdot\mathcal{X}(\chi;\theta).$$
\end{lemma}

Next, we record a consequence of the basic properties of characters of wreath products. A detailed proof of the following lemma can be found in \cite[Lemma 2.18]{SLThesis}.

\begin{lemma}\label{lem: 9.5}
Let $p$ be an odd prime and $G$ be a finite group. 
Let $\eta\in\Char(G)$ and $\varphi\in\Irr(G)$. If $\langle \eta,\varphi\rangle\ge 2$, then
$ \langle \mathcal{X}(\eta;\tau), \mathcal{X}(\varphi;\theta)\rangle \ge 2$
for all $\tau,\theta\in\Irr(C_p)$.
\end{lemma}

We conclude with a result that will be used frequently later in the article. 

\begin{lemma}\label{lem: 9.6}
Let $G$, $H$ be finite groups with $H\le\fS_m$ for some $m\in\mathbb{N}$, and let $\theta\in\Irr(H)$. Let $\alpha\in\Irr(G)$ and $\Delta\in\Char(G)$ be such that $\langle\Delta,\alpha\rangle=1$. Then for any $\beta\in\Irr(H)$,
$$\langle\mathcal{X}(\Delta;\theta), \mathcal{X}(\alpha;\beta)\rangle = \langle\theta,\beta\rangle = \delta_{\theta,\beta}.$$
\end{lemma}
\begin{proof}
Let $\zeta = \mathcal{X}(\Delta;\theta)$ and $c=\langle\zeta,\mathcal{X}(\alpha;\theta)\rangle$. Clearly $c\ge 1$, since $\alpha$ is a constituent of $\Delta$. (More generally, if $\Delta=\psi_1+\cdots+\psi_r$ is a decomposition into irreducible constituents, then $\sum_i \mathcal{X}(\psi_i;\theta)$ is a direct summand of $\mathcal{X}(\Delta;\theta)$.) Now $\zeta\down_{G^{\times m}}=\theta(1)\cdot\Delta^{\times m}$, so
\begin{align*}
\theta(1) &= \theta(1)\cdot(\langle\Delta,\alpha\rangle)^m = \langle\zeta\down_{G^{\times m}},\alpha^{\times m}\rangle = \sum_{\gamma\in\Irr(G\wr H)}\langle\zeta,\gamma\rangle\cdot\langle\gamma\down_{G^{\times m}},\alpha^{\times m}\rangle\\
&\ge \sum_{\beta\in\Irr(H)}\langle\zeta,\mathcal{X}(\alpha;\beta)\rangle\cdot\langle\mathcal{X}(\alpha;\beta)\down_{G^{\times m}},\alpha^{\times m}\rangle =\sum_{\beta\in\Irr(H)} \beta(1)\cdot\langle\zeta,\mathcal{X}(\alpha;\beta)\rangle\\
&\ge \theta(1)\cdot c \ge \theta(1).
\end{align*}
Thus the above inequalities in fact hold with equality and the claim follows.
\end{proof}

\subsection{The Sylow $p$-subgroups of $\fS_n$}\label{sec: sylow prelims}
We recall some facts about Sylow subgroups of symmetric groups, and refer the reader to \cite[Chapter 4]{JK} for a more detailed discussion. 
In doing so, we also fix the notation necessary to state Main Results A and B from the introduction. These are Theorems A and B below.  

Fix a prime $p$, an integer $n\in\mathbb{N}$ and $P_n\in\Syl_p(\fS_n)$.
Clearly $P_1$ is the trivial group while $P_p$ is cyclic of order $p$. More generally, $P_{p^k}= (P_{p^{k-1}})^{\times p}\rtimes P_p=P_{p^{k-1}}\wr P_p\cong P_p\wr \cdots \wr P_p$ ($k$-fold wreath product) for all $k\in\mathbb{N}$.
If $n\in\mathbb{N}$ has $p$-adic expansion $n=\sum_{i=1}^t a_ip^{n_i}$, that is, $t\in\mathbb{N}$, $0\le n_1<\cdots<n_t$ and $a_i\in[p-1]$ for all $i$, then $P_n\cong (P_{p^{n_1}})^{\times a_1}\times\cdots\times (P_{p^{n_t}})^{\times a_t}$.

Next, we fix a parametrisation of the linear characters of $P_n$, for all $n$.
We begin with the case when $n$ is a power of $p$.
Let $\Irr(P_p)=\{\phi_0, \phi_1,\dotsc, \phi_{p-1}\}=\Lin(P_p)$, where $\phi_0=\triv_{P_p}$ is the trivial character of the cyclic group $P_p$. 
When $k\geq 2$, \cite[Corollary 6.17]{IBook} shows that $$\mathrm{Lin}(P_{p^k})=\bigsqcup_{\phi\in\mathrm{Lin}(P_{p^{k-1}})}\Irr(P_{p^k}\ |\ \phi^{\times p}).$$
In particular,
$\Irr(P_{p^k}\ |\ \phi^{\times p})= \{\mathcal{X}(\phi;\psi)\ |\ \psi\in\mathrm{Lin}(P_p)\}.$

\smallskip

Using the above observations, we may naturally define a bijection $s\longleftrightarrow\phi(s)$ between the set $[\overline{p}]^k$ of sequences of length $k$ with elements from $[\overline{p}]$ and the set $\mathrm{Lin}(P_{p^k})$. More precisely, if $k=0$ we let the empty sequence of length 0 correspond to the trivial character of $P_1$, and if $k=1$ we let $s=(x)$ correspond to $\phi_x$, for each $x\in [\overline{p}]$. If $k\geq 2$ then for any $s=(s_1,\ldots, s_k)\in [\overline{p}]^k$, we recursively define 
$$\phi(s):=\mathcal{X}\big(\phi(s^-); \phi(s_k)\big),$$ 
where $s^-=(s_1,\ldots, s_{k-1})\in [\overline{p}]^{k-1}$. We observe that for any $i\in[k-1]$, Lemma~\ref{lem:assoc} guarantees that $\phi(s)=\mathcal{X}(\phi(s_1,\dotsc,s_i);\phi(s_{i+1},\dotsc,s_k))$. 
Moreover, under this labelling the trivial character $\triv_{P_{p^k}}$ corresponds to the sequence $(0,\dotsc,0)\in[\overline{p}]^k$. 

\smallskip

We make the following remark for the interested reader: once we fix a natural isomorphism $P_{p^k}/P_{p^k}' \cong(C_p)^k$, our indexing of $\Lin(P_{p^k})$ can in fact be obtained equivalently from the canonical bijection $\Lin(P_{p^k}) \longleftrightarrow \mathrm{Irr}(P_{p^k}/P_{p^k}')$. This correspondence is described in detail in \cite{GLL}.

\smallskip

We can now introduce the notation necessary to state our main results.
\begin{defn}\label{def:f and eta}
Let $k\in\mathbb{N}$ and $s\in[\overline{p}]^k$.
\begin{itemize}\setlength\itemsep{0.5em}
\item[$\circ$] For $z\in\{0,1,\dotsc,k\}$, let $U_k(z)=\{s\in[\overline{p}]^k : |\{i\in[k] : s_i\ne 0\}| = z \}$.
\item[$\circ$] If $s\in U_k(z)$ where $z\ge 1$, then define $f(s)=\min\{i\in[k] \mid s_i\ne 0\}$. 
\item[$\circ$] If $s\in U_k(z)$ where $z\ge 2$, then define $g(s)=\min\{i>f(s) \mid s_i\ne 0 \}$. 
\end{itemize}
Notice that $f(s)$ and $g(s)$ are just the positions of the leftmost and second leftmost non-zero entries in the sequence $s$.
\end{defn}

Next, we divide sequences of length $k$ into types.

\begin{defn}\label{def: m(s) types}
Let $k\in\mathbb{N}$ and $s\in[\overline{p}]^k$. The \emph{type} of $s$ is the number $\tau(s)\in\{1,2,3,4\}$ defined as follows:
	$$\tau(s) = 
	\begin{cases} 
	1 & \mathrm{if}\ s=(0,0,\ldots, 0), \\ 
	2 & \mathrm{if}\ s\in U_k(1)\ \mathrm{and}\ f(s)<k,\\
	3 & \mathrm{if}\ s\in U_k(1)\ \mathrm{and}\ f(s)=k,\\
	4 & \mathrm{if}\ s\in U_k(z)\ \mathrm{for\ some}\ z\ge 2. \end{cases}$$
\end{defn}

\smallskip

Now let $n\in\mathbb{N}$ be arbitrary. Suppose $n$ has $p$-adic expansion $n=\sum_{i=1}^t a_ip^{n_i}$ where $0\le n_1<\cdots<n_t$ and $a_i\in[p-1]$.
Since $P_n\cong (P_{p^{n_1}})^{\times a_1}\times\cdots\times (P_{p^{n_t}})^{\times a_t}$, then
\begin{equation}\label{eqn: index}
\Lin(P_n) = \{\phi(\underline{\mathbf{s}})\ |\ \underline{{\mathbf{s}}}=\big(\mathbf{s}(1,1),\ldots, \mathbf{s}(1,a_1), \mathbf{s}(2,1),\ldots, \mathbf{s}(2,a_2), \ldots, \mathbf{s}(t,a_t)\big) \},
\end{equation}
where for all $i\in[t]$ and $j\in[a_i]$ we have that ${\bf{s}}(i,j)\in[\overline{p}]^{n_i}$, and 
$$\phi(\underline{\mathbf{s}}) := \phi(\mathbf{s}(1,1))\times\cdots\times\phi(\mathbf{s}(1,a_1))\times \phi(\mathbf{s}(2,1))\times\cdots\times\phi(\mathbf{s}(2,a_2))\times\cdots\times\phi(\mathbf{s}(t,a_t)).$$ 
When we write that $\phi(\underline{\mathbf{s}})$ is a linear character of $P_n$, we mean that $\underline{\mathbf{s}}$ is a sequence of sequences of the form described in (\ref{eqn: index}) above. 
To simplify notation, we let $R=\sum_{i=1}^t a_i$ and let $\{s_1,\dotsc,s_R\}$ be the multiset defined by $\{s_1,\dotsc,s_R\} = \{\mathbf{s}(i,j) \mid i\in[t],\ j\in[a_i]\}.$ In this case we say that $\phi$ corresponds to $\{s_1,\dotsc,s_R\}$.

The following definitions are crucial for our main theorems.

\begin{defn}\label{def: typephi}
	Let $n\in\mathbb{N}$. Let $\phi\in\mathrm{Lin}(P_n)$ and suppose it corresponds to the multiset of sequences $\{s_1,\ldots, s_R\}$. The \textit{type} of $\phi$ is the $4$-tuple $T(\phi)=(x_1,x_2,x_3,x_4)$, where for each $i\in [4]$ we set 
	$$x_i:=|\{j\in [R]\ |\ \tau(s_j)=i\}|.$$
\end{defn}

\smallskip

\begin{defn}\label{def: quasi-trivial}
Let $n\in\mathbb{N}$ and $\phi\in\mathrm{Lin}(P_n)\setminus\{1_{P_n}\}$. Suppose $\phi$ corresponds to the multiset of sequences $\{s_1,\ldots, s_R\}$. We say that $\phi$ is \textit{quasi-trivial} if there is at most one non-zero entry in each component sequence $s_i$.
In other words, $\phi$ is quasi-trivial if $\tau(s_i)\ne 4$ for all $i\in[R]$. Equivalently, $\phi$ is quasi-trivial if $T(\phi)=(x_1,x_2,x_3,0)$ for some $x_1,x_2,x_3\in\mathbb{N}_0$.
\end{defn}

Recall the definitions of the quantities $m(\phi)$ and $M(\phi)$ from the introduction:
$$m(\phi) = \max\{t\in\mathbb{N} \mid \mathcal{B}_{n}(t) \subseteq\Omega(\phi)\}\quad \text{and}\quad  M(\phi) = \min\{t\in\mathbb{N} \mid \Omega(\phi)\subseteq\mathcal{B}_{n}(t)\},$$
where $\mathcal{B}_n(t)=\{\lambda\in\mathcal{P}(n) : \lambda_1\le t,\ l(\lambda)\le t \}$.
We are now ready to give the precise statement of our Main Result A from the introduction.

\begin{thmA}
Let $n$ be a natural number and $p\ge 5$ be a prime. Let $\phi\in\mathrm{Lin}(P_n)\setminus\{\triv_{P_n}\}$ be quasi-trivial. Then $\Omega(\phi)=\mathcal{B}_{n}(m(\phi)),$
unless $T(\phi)=(R-1,1,0,0)$, in which case
	$$\Omega(\phi) = \mathcal{B}_n(m(\phi)) \sqcup\{(m(\phi)+1,\mu)\ |\ \mu\in\Omega(\triv_{P_{n-(m(\phi)+1)}}) \}^\circ.$$
\end{thmA}

Theorem A provides a complete description of $\Omega(\phi)$ for all quasi-trivial characters $\phi$, because $m(\phi)$ is explicitly determined. We omitted this in the statement of Theorem A as we are going to compute $m(\phi)$ for all $\phi\in\mathrm{Lin}(P_n)$ in Theorem B below. 

\smallskip

Given $\phi\in\mathrm{Lin}(P_{p^k})$ corresponding to $s\in [\overline{p}]^k$, we sometimes denote $m(\phi)$ and $M(\phi)$ by $m(s)$ and $M(s)$ respectively. 

\begin{defn}\label{def: new omegas N}
	Let $k\in\mathbb{N}_0$ and $s\in[\overline{p}]^k$. The integer $N(s)$ is defined as follows:
	$$N(s) = \begin{cases}
	p^k & \mathrm{if}\ \tau(s)=1,\\
	m(s)+1 & \mathrm{if}\ \tau(s)=2,\\
	m(s) & \mathrm{if}\ \tau(s)\in\{3,4\}.
	\end{cases}$$
\end{defn}

We are now ready to give the precise statement of Main Result B. 

\begin{thmB}
	Let $p\ge 5$ be a prime. Let $k\in\mathbb{N}$, and suppose $\phi=\phi(s)\in\Lin(P_{p^k})\setminus\{\triv_{P_{p^k}}\}$. Then $M(\phi) = p^k-p^{k-f(s)}$, and
	$$m(\phi) = \begin{cases} 
	p^k-p^{k-f(s)}-1 +\delta_{f(s),k} & \mathrm{if}\ \tau(s)\neq 4,\\ 
	p^k-p^{k-f(s)}-p^{k-g(s)} & \mathrm{if}\ \tau(s)= 4.
	\end{cases}$$
	
	Let $n\in\mathbb{N}$, not a power of $p$, and suppose it has $p$-adic expansion $n=\sum_{i=1}^t a_ip^{n_i}$. 
	For $\phi\in\Lin(P_n)$ with corresponding multiset $\{s_1,\ldots, s_R\}$, we have
	$$M(\phi) = \sum_{i=1}^R M(s_i) \qquad\mathrm{and}\qquad m(\phi) = \sum_{i=1}^R N(s_i).$$
\end{thmB}

Here $\delta$ denotes the Kronecker delta.
Notice also that for all odd primes $p$ and $n\in\mathbb{N}$, $m(\triv_{P_n})$ and $M(\triv_{P_n})$ have already been determined in \cite{GL1}.

\begin{rem}\label{rem: N}
Let $n$ be a natural number, $p$ a prime and $N=N_{\frak{S}_n}(P_n)$. 
The normaliser $N$ naturally acts by conjugation on the set $\mathrm{Lin}(P_n)$. If $\phi,\psi\in\mathrm{Lin}(P_n)$ are $N$--conjugate, then clearly $\phi\up^{\fS_n}=\psi\up^{\fS_n}$.
Extending a result of Navarro \cite{N}, we have recently shown in \cite{GLL} that the converse holds in the case of symmetric groups. Namely, given $\phi,\psi\in\mathrm{Lin}(P_n)$, then $\phi\up^{\fS_n}=\psi\up^{\fS_n}$ if and only if $\phi$ and $\psi$ are $N$--conjugate.

This result transcends the purpose of the present paper, and more importantly, it does not simplify our study of the sets $\Omega(\phi)$. For this reason we only briefly mention the structure (and labelling) of the $N$--orbits on $\mathrm{Lin}(P_n)$ in Remark~\ref{rem: N2} below.\hfill$\lozenge$
\end{rem}

\begin{rem}\label{rem: N2}
Let $n=p^k$ and let $s,t\in [\overline{p}]^k$. It is not too difficult to see that 
the corresponding linear characters $\phi(s)$ and $\phi(t)$ lie in the same $N$--orbit if and only if 
$$\{i\in [k]\ |\ s_i\neq 0\}=\{i\in [k]\ |\ t_i\neq 0\}.$$
In this case we write $s\equiv t$.
In particular, the set $\{0,1\}^k\subseteq [\overline{p}]^k$ labels a set of representatives for the orbits of $N$ in $\mathrm{Lin}(P_n)$.
Notice that this observation is reflected in the statements of our main results: we always consider the position of certain non-zero entries in a given sequence (i.e. $f(s), g(s)$), but we never give importance to the actual value of such entries. 

In general, let $n\in\mathbb{N}$ have $p$-adic expansion $n=\sum_{i=1}^t a_ip^{n_i}$ and let $\phi,\psi\in\mathrm{Lin}(P_n)$ be such that $$\phi = \phi(\mathbf{s}(1,1))\times\cdots\times\phi(\mathbf{s}(1,a_1))\times \phi(\mathbf{s}(2,1))\times\cdots\times\phi(\mathbf{s}(2,a_2))\times\cdots\times\phi(\mathbf{s}(k,a_k)),$$
$$\psi = \phi(\mathbf{t}(1,1))\times\cdots\times\phi(\mathbf{t}(1,a_1))\times \phi(\mathbf{t}(2,1))\times\cdots\times\phi(\mathbf{t}(2,a_2))\times\cdots\times\phi(\mathbf{t}(k,a_k)),$$  
for appropriate sequences $\mathbf{s}(i,j),\mathbf{t}(i,j)$ as described in \eqref{eqn: index}. Then $\phi$ and $\psi$ are $N$--conjugate if and only if for all $i\in[k]$, there exists $\sigma\in\mathfrak{S}_{a_i}$ such that $\mathbf{s}(i,j)\equiv\mathbf{t}(i,\sigma(j))$ for all $j\in[a_i]$.
Proofs of the aforementioned observations can be found in \cite{GLL}.\hfill$\lozenge$
\end{rem}

\medskip

\subsection{The representation theory of $\fS_n$}\label{sec: sn prelims}
For each $n\in\mathbb{N}$, $\Irr(\fS_n)$ is naturally in bijection with the set $\mathcal{P}(n)$ of all partitions of $n$. For a partition $\lambda\in\mathcal{P}(n)$, also written $\lambda\vdash n$, we denote the corresponding irreducible character of $\fS_n$ by $\chi^\lambda$. Under this natural bijection, the trivial character of $\fS_n$ corresponds to $(n)$, and the sign or alternating character to $(1^n)$ \cite[2.1.7]{JK}. When clear from context, we also abbreviate wreath product characters $\mathcal{X}(\chi^\lambda;\chi^\mu)$ involving characters of symmetric groups to simply $\mathcal{X}(\lambda;\mu)$.

\smallskip

For a partition $\lambda$, its size, length (number of parts) and conjugate are denoted by $|\lambda|$, $l(\lambda)$ and $\lambda'$ respectively. 
Given two partitions $\lambda$ and $\mu$ we denote by $\lambda+\mu$ the partition whose $i$-th part is given by $\lambda_i+\mu_i$, for all $i\in\mathbb{N}$ (here we regard $\lambda_i=0$, whenever $i>l(\lambda)$). 
The Young diagram $[\lambda]$ corresponding to the partition $\lambda=(\lambda_1,\lambda_2,\dotsc,\lambda_k)$ is the subset of the Cartesian plane defined by: 
$$[\lambda]=\{(i,j)\in\mathbb{N}\times\mathbb{N}\ |\ 1\leq i\leq k,\ 1\leq j\leq\lambda_i\},$$
where we view the diagram in matrix orientation, with the node $(1,1)$ in the upper left corner.
To be precise, the node or box $(i,j)$ is that lying in row $i$ and column $j$.
Given two partitions $\lambda$ and $\mu$ we say that $\mu\subseteq \lambda$ if $[\mu]\subseteq [\lambda]$.
Given $\lambda\in\mathcal{P}(n)$ we sometimes denote by $\lambda^-$ the subset of $\mathcal{P}(n-1)$ consisting of those partitions $\mu$ such that $\mu\subseteq \lambda$. Similarly $\lambda^+$ is the set of partitions of $n+1$ whose Young diagrams contain $[\lambda]$. 

As usual, we say that $\lambda\in\mathcal{P}(n)$ is a \textit{hook partition} if $\lambda=(n-x, 1^x)$ for some $0\leq x\leq n-1$. 
We introduce the following definition which will play an important role in this article. 

\begin{defn}\label{def:thin}
	We say that a partition $\lambda$ is \emph{thin} if $\lambda$ is a hook, $l(\lambda)\le 2$, or $\lambda_1\le 2$.
	
	For $m\le n\in\mathbb{N}$, we denote the hook partition $(m,1^{n-m})$ by $\hook{n}{m}$. When $m\ge\tfrac{n}{2}$, we denote the partition $(m,n-m)$ by $\tworow{n}{m}$. 
\end{defn}

We record some easy and useful facts.

\begin{lemma}\label{lem: MN}
	Let $p$ be a prime, let $P\in\Syl_p(\fS_p)$ and let $\psi$ be the regular character of $P$.
	Then
	$$\chi^\lambda\down^{\fS_p}_P = \begin{cases}
	m\cdot\psi & \mathrm{if}\ \lambda\ \mathrm{is\ not\ a\ hook},\\
	m'\cdot\psi+(-1)^l\cdot\triv_P & \mathrm{if}\ \lambda\ \mathrm{is\ a\ hook\ of\ leg\ length}\ l,\end{cases}$$
	where $m$ and $m'$ are integers satisfying $mp=\chi^\lambda(1)$ and $m'p+(-1)^l = \chi^\lambda(1)$ in the respective cases.
\end{lemma}

\begin{proof}
Since $P$ is a cyclic group generated by a cycle of length $p$, the statement follows from the Murnaghan--Nakayama Rule \cite[2.4.7]{JK}.
\end{proof}

\medskip

\subsection{The Littlewood\textendash Richardson Rule}\label{sec: LR prelims}
Let $m,n\in\mathbb{N}$ with $m<n$. For $\mu\vdash m$ and $\nu\vdash n-m$, the Littlewood\textendash Richardson rule (see \cite[Chapter 16]{J}) describes the decomposition into irreducible constituents of induced character
$$(\chi^\mu\times\chi^\nu)\up^{\fS_n}_{\fS_m\times \fS_{n-m}}.$$

Before we recall the Littlewood\textendash Richardson rule, we introduce some notation and technical definitions. By a skew shape $\gamma$ we mean a set difference of Young diagrams $[\lambda\setminus\mu]:=[\lambda]\setminus[\mu]$ for some partitions $\lambda$ and $\mu$ with $[\mu]\subsetneq[\lambda]$, and $|\gamma|:=|\lambda|-|\mu|$. By convention, the highest row of $[\lambda]$ for a partition $\lambda$ is numbered 1, but the highest row of a skew shape $\gamma=[\lambda\setminus\mu]$ need not be the highest row of $[\lambda]$.

\begin{defn}\label{def: LR type}
	Let $\lambda=(\lambda_1,\dotsc,\lambda_k)\in\mathcal{P}(n)$ and let $\mathcal{C}=(c_1,\dotsc,c_n)$ be a sequence of positive integers. We say that $\mathcal{C}$ is of \emph{weight $\lambda$} if
	$$|\{i\in\{1,\dotsc,n\}\ :\ c_i=j\}| = \lambda_j\ , \ \text{for all}\ \ j\in\{1,\dotsc,k\}.$$
	
	We say that an element $c_j$ of $\mathcal{C}$ is \emph{good} if $c_j=1$ or if
	$$|\{i\in\{1,2,\dotsc,j-1\}\ :\ c_i=c_j-1\}|>|\{i\in\{1,2,\dots,j-1\}\ :\ c_i=c_j\}|.$$
	Finally, we say that the sequence $\mathcal{C}$ is \emph{good} if $c_j$ is good for every $j\in\{1,\dotsc,n\}$.
\end{defn}

\begin{thm}[Littlewood\textendash Richardson rule]\label{thm:LR}
	Let $m,n\in\mathbb{N}$ with $m<n$. Let $\mu\vdash m$ and $\nu\vdash n-m$. Then
	$$(\chi^\mu\times\chi^\nu)\up^{\fS_n}_{\fS_m\times \fS_{n-m}} = \sum_{\lambda\vdash n} c^\lambda_{\mu\nu}\ \chi^\lambda$$
	where $c^\lambda_{\mu\nu}$ equals the number of ways to replace the nodes of $[\lambda\setminus\mu]$ by natural numbers such that
	\begin{itemize}\setlength\itemsep{0.5em}
		\item[(i)] the sequence obtained by reading the numbers from right to left, top to bottom is a good sequence of weight $\nu$;
		\item[(ii)] the numbers are weakly increasing along rows; and
		\item[(iii)] the numbers are strictly increasing down columns.
	\end{itemize}
\end{thm}

Let $\nu$ be a partition. We call a way of replacing the nodes of a skew shape $\gamma$ with $|\nu|$ boxes by numbers satisfying conditions (i)--(iii) of Theorem~\ref{thm:LR} a \emph{Littlewood\textendash Richardson filling of $\gamma$ of weight $\nu$}. It is easy to see that every skew shape has at least one Littlewood\textendash Richardson filling. Moreover, the coefficients described in Theorem~\ref{thm:LR} are symmetric: $c^\lambda_{\mu\nu}=c^\lambda_{\nu\mu}$ for all partitions $\mu$, $\nu$ and all partitions $\lambda\vdash |\mu|+|\nu|$.
Let $\mathcal{LR}(\gamma)$ denote the set of all possible weights of Littlewood\textendash Richardson fillings of a skew shape $\gamma$. 

\smallskip

The following necessary condition for positivity of Littlewood--Richardson coefficients is an immediate consequence of Theorem~\ref{thm:LR}.

\begin{lemma}\label{lem: LRfirstpart}
	Let $\mu$ and $\nu$ be partitions. Let $\lambda$ be a partition of $|\mu|+|\nu|$ and suppose that $c^\lambda_{\mu\nu}>0$. Then $\lambda_1\le \mu_1+\nu_1$ and $l(\lambda)\le l(\mu)+l(\nu)$.
\end{lemma}

We can also define \emph{iterated Littlewood--Richardson coefficients} $c^{\lambda}_{\mu^1,\dotsc,\mu^r}$ as follows. Let $r\in\mathbb{N}$ and $\mu^1,\dotsc,\mu^r$ be partitions, and let $\lambda\vdash n:= |\mu^1|+\cdots+|\mu^r|$. Then $c^{\lambda}_{\mu^1,\dotsc,\mu^r}$ is the multiplicity of $\chi^\lambda$ as a constituent of $(\chi^{\mu^1}\times\cdots\times\chi^{\mu^r}) \up^{\fS_n}_{\fS_{|\mu^1|}\times\cdots\times\fS_{|\mu^r|}}$. When $r=2$, these are the usual Littlewood--Richardson coefficients as defined above. Letting $m=|\mu^1|+\cdots+|\mu^{r-1}|$ when $r\ge 2$, it is easy to see that

$$c^{\lambda}_{\mu^1,\dotsc,\mu^r} = \sum_{\gamma\vdash m} c^{\gamma}_{\mu^1,\dotsc,\mu^{r-1}} \cdot c^\lambda_{\gamma,\mu^r}.$$
The iterated Littlewood--Richardson coefficients are also symmetric under any permutation of the partitions $\mu^1,\dotsc,\mu^r$.
An iterated Littlewood--Richardson filling of $[\lambda]$ by $\mu^1,\dotsc,\mu^r$ is a way of replacing the nodes of $[\lambda]$ by numbers defined recursively as follows: if $r=1$ then $[\lambda]=[\mu^1]$ has a unique Littlewood--Richardson filling, of weight $\mu^1$; if $r\ge 2$ then we mean an iterated Littlewood--Richardson filling of $[\gamma]$ by $\mu^1,\dotsc,\mu^{r-1}$ together with an Littlewood--Richardson filling of $[\lambda\setminus\gamma]$ of weight $\mu^r$ (for some $\gamma\subseteq\lambda$ such that this is possible).

\begin{lemma}\label{lem: iteratedLR}
	Let $a,b_1,\dotsc,b_a\in\mathbb{N}$. Let $\nu^1,\dotsc,\nu^a$ be partitions such that $b_i\ge |\nu^i|$ for all $i$ and let $c=|\nu^1|+\cdots+|\nu^a|$. 
	Let $\mu\vdash c$ and let $\lambda=(b_1+b_2+\cdots+b_a,\mu)$. Then the iterated Littlewood--Richardson coefficients $c^\lambda_{(b_1,\nu^1),\dotsc,(b_1,\nu^a)}$ and $c^\mu_{\nu^1,\dotsc,\nu^a}$ are equal.
\end{lemma}

\begin{proof}
	Clearly $c^\mu_{\nu^1,\dotsc,\nu^a}\le c^\lambda_{(b_1,\nu^1),\dotsc,(b_1,\nu^a)}$, since we may take any Littlewood--Richardson filling of $[\mu]$ by $\nu^1,\dotsc,\nu^a$ and replace each number $i$ by $i+1$, then combine with the first row of $[\lambda]$ filled with all 1s to produce a Littlewood--Richardson filling of $[\lambda]$ by $(b_1,\nu^1),\dotsc,(b_a,\nu^a)$. Conversely, any such filling of $[\lambda]$ contains 1s in exactly the first row of $[\lambda]$ since $\lambda_1=b_1+\cdots+b_a$, so this process is bijective. Thus $c^\mu_{\nu^1,\dotsc,\nu^a}= c^\lambda_{(b_1,\nu^1),\dotsc,(b_1,\nu^a)}$.
\end{proof}

We conclude this section by introducing an operator that will be useful later.

\begin{defn}\label{def:star operator}
	For $n,m\in\mathbb{N}$ and $A\subseteq\mathcal{P}(n)$, $B\subseteq\mathcal{P}(m)$, let 
	$$A\star B:=\{\lambda\vdash n+m \mid \exists\ \mu\in A,\ \nu\in B\ \mathrm{such\ that}\ c^{\lambda}_{\mu\nu}>0 \}.$$
\end{defn}
It is routine to check that $\star$ is both commutative and associative.

\medskip

\section{Preliminary combinatorial results}\label{sec:LR}

In this section we prove some results concerning Littlewood--Richardson coefficients useful for later sections, which may also be of independent interest.

\begin{defn}\label{def:bnm}
	Suppose $A\subseteq\bigcup_{n\in\mathbb{N}} \mathcal{P}(n)$. We define $A':=\{\lambda'\ |\ \lambda\in A\}$ and $A^\circ:=A\cup A'$. 
\end{defn}

Recall for $n,t\in\mathbb{N}$ that 
$$\mathcal{B}_n(t) = \{\lambda\vdash n\ |\ \lambda_1\le t\ \mathrm{and}\ l(\lambda)\le t \}.$$
Thus $\mathcal{B}_n(t)$ is the set of those partitions of $n$ whose Young diagrams fit inside an $t\times t$ square grid. Notice that $\mathcal{B}_n(t)$ is closed under taking conjugates of partitions, i.e.~$\mathcal{B}_n(t)^\circ = \mathcal{B}_n(t)$. 

\smallskip

Our first aim is to show that under appropriate hypotheses on the parameters, we have $\mathcal{B}_n(t) \star\mathcal{B}_{n'}(t') = \mathcal{B}_{n+n'}(t+t')$. This is proved in Proposition~\ref{prop: B star B} below using an inductive argument.
The following lemma deals with the base step of our induction. 

\begin{lemma}\label{lem: LR base case}
	Let $t,t'\in\mathbb{N}$. Then $\mathcal{B}_{2t-1}(t)\star \mathcal{B}_{2t'-1}(t') = \mathcal{B}_{2t+2t'-2}(t+t')$.
\end{lemma}

\begin{proof}
	That $\mathcal{B}_{2t-1}(t)\star \mathcal{B}_{2t'-1}(t') \subseteq \mathcal{B}_{2t+2t'-2}(t+t')$ follows from Definition~\ref{def:star operator} and Lemma~\ref{lem: LRfirstpart}. 
	To prove the converse, we proceed by induction on $t+t'$. The base case follows from the observation that for any natural numbers $N$ and $M$ such that $N<2M$, we have $\mathcal{B}_N(M)\star\mathcal{B}_1(1)\supseteq \mathcal{B}_{N+1}(M+1)$: given any partition $\lambda\in\mathcal{B}_{N+1}(M+1)$, either $\lambda\in\mathcal{B}_{N+1}(M)$ and we clearly have that $\lambda\in\mathcal{B}_N(M)\star\mathcal{B}_1(1)$; or $\lambda_1=M+1$, in which case $\lambda_2<M+1$ since $N<2M$, 
	and so considering $\mu=(\lambda_1-1,\lambda_2,\dotsc)\in\lambda^-$ we obtain that $\lambda\in\mathcal{B}_N(M)\star\mathcal{B}_1(1)$ (the case where $l(\lambda)=M+1$ is dealt with similarly). 
	
	We may now assume that $t,t'\ge 2$. For the inductive step, we take as inductive hypothesis $\mathcal{B}_{2t-3}(t-1)\star\mathcal{B}_{2t'-1}(t')=\mathcal{B}_{2t+2t'-4}(t+t'-1)$. By applying $-\star\mathcal{B}_1(1)$ to both sides, 
	we find
	$$\mathcal{B}_{2t-3}(t-1)\star\mathcal{B}_{2t'}(t'+1)=\mathcal{B}_{2t+2t'-3}(t+t'),$$ 
	and then applying $\mathcal{B}_1(1)\star-$ to both sides, 
	we find
	$$\mathcal{B}_{2t-2}(t)\star\mathcal{B}_{2t'}(t'+1)=\mathcal{B}_{2t+2t'-2}(t+t'+1).$$
	Hence
	$$\mathcal{B}_{2t+2t'-2}(t+t')\subseteq\mathcal{B}_{2t+2t'-2}(t+t'+1) = \mathcal{B}_{2t-2}(t)\star\mathcal{B}_{2t'}(t'+1).$$
	Thus, letting $\lambda\in\mathcal{B}_{2t+2t'-2}(t+t')$, there exist partitions $\mu\in\mathcal{B}_{2t-2}(t)$ and $\nu\in\mathcal{B}_{2t'}(t'+1)$ such that $c^\lambda_{\mu\nu}>0$. Fix a Littlewood--Richardson filling $\mathsf{F}$ of weight $\nu$ of the skew shape $[\lambda\setminus\mu]$. 
	
	\smallskip
	
	To complete the inductive step, we construct $\hat{\mu}\in\mathcal{B}_{2t-1}(t)$ and $\hat{\nu}\in\mathcal{B}_{2t'-1}(t')$ such that $c^\lambda_{\hat{\mu}\hat{\nu}}>0$, from which we conclude therefore that $\mathcal{B}_{2t+2t'-2}(t+t')\subseteq\mathcal{B}_{2t-1}(t)\star \mathcal{B}_{2t'-1}(t')$. The main idea is to remove an appropriate box $\mathsf{b}$ from the skew shape $[\lambda\setminus\mu]$, set $[\hat{\mu}]=[\mu]\cup\mathsf{b}$, and exhibit an appropriate filling $\mathsf{F'}$ of $[\lambda\setminus\hat{\mu}]$ of weight $\hat{\nu}$. 
	
	\smallskip
	
	Since all sets considered are closed under conjugation of partitions, we may without loss of generality assume $\nu_1\ge l(\nu)$ (by taking $\lambda'$, $\mu'$ and $\nu'$ instead of $\lambda$, $\mu$ and $\nu$ if necessary). Let $k\ge 1$ be such that $\nu_1=\nu_2=\dotsc=\nu_k>\nu_{k+1}$, and let $\mathsf{x}$ denote the box containing the last 1 in the Littlewood--Richardson reading order of the filling $\mathsf{F}$ (namely right to left, top to bottom). Clearly this must lie at the top of its column and leftmost in its row in $[\lambda\setminus\mu]$, and so must be an addable box for $\mu$. 
	We split into three cases according to the position of $\mathsf{x}$.
	
	\smallskip
	
	\noindent\emph{Case (i): if the position of $\mathsf{x}$ is neither $(1,t+1)$ nor $(t+1,1)$.} 
	Since $\mathsf{x}$ is an addable box for $\mu\in\mathcal{B}_{2t-2}(t)$, setting $\hat{\mu}$ to be the partition whose Young diagram is $[\mu]\cup\mathsf{x}$ we find that $\hat{\mu}\in\mathcal{B}_{2t-1}(t)$. 
	
	If $k=1$ then the filling $\mathsf{F'}$ defined as $\mathsf{F}$ restricted to the boxes of $[\lambda\setminus\hat{\mu}]$ is a Littlewood--Richardson filling 
	of weight $\hat{\nu}:=(\nu_1-1,\nu_2,\dotsc,\nu_{l(\nu)})\in\mathcal{B}_{2t'-1}(t'+1)$. 
	Moreover, $\hat{\nu}_1=\nu_1-1\le t'+1-1=t'$, and $l(\hat{\nu})=l(\nu)\le t'$ since $l(\nu)\le \nu_1$ and $|\nu|=2t'$. Thus $\hat{\nu}\in\mathcal{B}_{2t'-1}(t')$. 
	
	If $k>1$, then $(\nu_1-1,\nu_2,\dotsc,\nu_{l(\nu)})$ is not a partition: in this case we define $\mathsf{F'}$ and $\hat{\nu}$ as follows. Let $i\in\{2,\dotsc,k\}$ and consider the position of the last $i$ in the reading order of the filling $\mathsf{F}$. By the definition of Littlewood--Richardson fillings, the last $i$ must appear later in the reading order than the last $i-1$ since $\nu_{i-1}=\nu_i$. Since this holds for all $i\in\{2,\dotsc,k\}$, the box containing the last $i$ must be the leftmost $i$ in its row in $[\lambda\setminus\mu]$ (and hence leftmost in its row), and either at the top of its column or immediately below the box containing the last $i-1$ in the reading order of $\mathsf{F}$. 
	Thus we may define a Littlewood--Richardson filling $\mathsf{F'}$ of $[\lambda\setminus\hat{\mu}]$ to be obtained from $\mathsf{F}$ by removing the 1 corresponding to the box $\mathsf{x}$, then relabelling the last $i$ in $\mathsf{F}$ by the number $i-1$, for each $2\le i\le k$. In particular, the weight of $\mathsf{F'}$ is the partition $\hat{\nu}:=(\nu_1,\dotsc,\nu_{k-1},\nu_k-1,\nu_{k+1},\dotsc,\nu_{l(\nu)})$. Moreover, $k>1$ and $l(\nu)\le \nu_1$ imply that $\nu\in\mathcal{B}_{2t'}(t')$, 
	and hence $\hat{\nu}\in\mathcal{B}_{2t'-1}(t')$. An example is shown in Figure~\ref{fig: case (i)}.
	
	\begin{figure}
		\centering
		\begin{tikzpicture}[scale=0.85, every node/.style={scale=0.8}]
		\draw [step=0.5] (0,0) grid (4.5,-0.5);
		\draw [step=0.5] (0,-0.5) grid (4,-1);
		\draw [step=0.5] (0,-1) grid (3.5,-1.5);
		\draw [step=0.5] (0,-1.5) grid (2.5,-2);
		\draw [step=0.5] (0,-2) grid (1.5,-2.5);
		
		\draw [line width=0.5mm] (0,0) -- (3,0);
		\draw [line width=0.5mm] (3,0) -- (3,-0.5);
		\draw [line width=0.5mm] (1.5,-0.5) -- (3,-0.5);
		\draw [line width=0.5mm] (1.5,-0.5) -- (1.5,-1.5);
		\draw [line width=0.5mm] (1.5,-1.5) -- (1,-1.5);
		\draw [line width=0.5mm] (1,-1.5) -- (1,-2);
		\draw [line width=0.5mm] (1,-2) -- (0,-2);
		\draw [line width=0.5mm] (0,0) -- (0,-2);
		
		\draw [pattern=north east lines, pattern color=darkgray] (1.5,-0.5) rectangle (2,-1);
		\draw (1.75,-1.25) circle [radius=0.17] node{};
		\draw (0.25,-2.25) circle [radius=0.17] node{};
		
		\draw (3.25,-0.25) node[] {$1$};
		\draw (3.75,-0.25) node[] {$1$};
		\draw (4.25,-0.25) node[] {$1$};
		\draw (1.75,-0.75) node[] {$1$};
		\draw (2.25,-0.75) node[] {$1$};
		\draw (2.75,-0.75) node[] {$2$};
		\draw (3.25,-0.75) node[] {$2$};
		\draw (3.75,-0.75) node[] {$2$};
		\draw (1.75,-1.25) node[] {$2$};
		\draw (2.25,-1.25) node[] {$2$};
		\draw (2.75,-1.25) node[] {$3$};
		\draw (3.25,-1.25) node[] {$3$};
		\draw (1.25,-1.75) node[] {$3$};
		\draw (1.75,-1.75) node[] {$3$};
		\draw (2.25,-1.75) node[] {$4$};
		\draw (0.25,-2.25) node[] {$3$};
		\draw (0.75,-2.25) node[] {$4$};
		\draw (1.25,-2.25) node[] {$5$};
		\end{tikzpicture}
		\hspace{10pt}
		\begin{tikzpicture}
		\draw (0,0.9) node[] {\LARGE$\leadsto$};
		\draw (0,0) node[] {};
		\end{tikzpicture}
		\hspace{10pt}
		\begin{tikzpicture}[scale=0.8, every node/.style={scale=0.8}]
		\draw [step=0.5] (0,0) grid (4.5,-0.5);
		\draw [step=0.5] (0,-0.5) grid (4,-1);
		\draw [step=0.5] (0,-1) grid (3.5,-1.5);
		\draw [step=0.5] (0,-1.5) grid (2.5,-2);
		\draw [step=0.5] (0,-2) grid (1.5,-2.5);
		
		\draw [line width=0.5mm] (0,0) -- (3,0);
		\draw [line width=0.5mm] (3,0) -- (3,-0.5);
		\draw [line width=0.5mm] (2,-0.5) -- (3,-0.5);
		\draw [line width=0.5mm] (2,-0.5) -- (2,-1);
		\draw [line width=0.5mm] (2,-1) -- (1.5,-1);
		\draw [line width=0.5mm] (1.5,-1) -- (1.5,-1.5);
		\draw [line width=0.5mm] (1.5,-1.5) -- (1,-1.5);
		\draw [line width=0.5mm] (1,-1.5) -- (1,-2);
		\draw [line width=0.5mm] (1,-2) -- (0,-2);
		\draw [line width=0.5mm] (0,0) -- (0,-2);
		
		\draw [pattern=north east lines, pattern color=darkgray] (1.5,-0.5) rectangle (2,-1);
		\draw [densely dotted, darkgray] (1.75,-1.25) circle [radius=0.17] node{};
		\draw [densely dotted, darkgray] (0.25,-2.25) circle [radius=0.17] node{};
		
		\draw (3.25,-0.25) node[] {$1$};
		\draw (3.75,-0.25) node[] {$1$};
		\draw (4.25,-0.25) node[] {$1$};
		\draw (2.25,-0.75) node[] {$1$};
		\draw (2.75,-0.75) node[] {$2$};
		\draw (3.25,-0.75) node[] {$2$};
		\draw (3.75,-0.75) node[] {$2$};
		\draw (1.75,-1.25) node[] {$1$};
		\draw (2.25,-1.25) node[] {$2$};
		\draw (2.75,-1.25) node[] {$3$};
		\draw (3.25,-1.25) node[] {$3$};
		\draw (1.25,-1.75) node[] {$3$};
		\draw (1.75,-1.75) node[] {$3$};
		\draw (2.25,-1.75) node[] {$4$};
		\draw (0.25,-2.25) node[] {$2$};
		\draw (0.75,-2.25) node[] {$4$};
		\draw (1.25,-2.25) node[] {$5$};
		\end{tikzpicture}
		\caption{\small{Example of case (i): $t=8$, $t'=9$, $\lambda=(9,8,7,5,3)\vdash 32$, $\mu=(6,3,3,2)$, $\nu=(5^3,2,1)$, $k=3$ and $\mathsf{F}$ as shown. On the left, the box $\mathsf{x}$ is shaded, and the last $i$ of $\mathsf{F}$ is circled for $2\le i\le k$. On the right, $\mathsf{F'}$ is shown with $\hat{\mu}=(6,4,3,2)$ and $\hat{\nu}=(5^2,4,2,1)$. The boxes containing the circled numbers have been relabelled to produce $\mathsf{F}'$.}} \label{fig: case (i)}
	\end{figure}
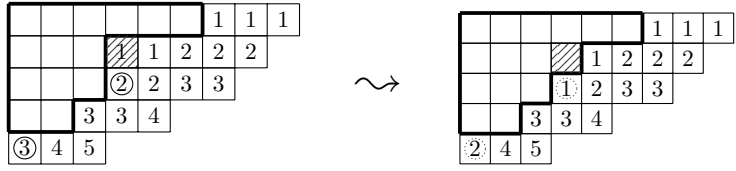
	
	Thus for all values of $k$, setting $\mathsf{b}=\mathsf{x}$ and taking $\hat{\mu}$, $\hat{\nu}$ as described above we find that $\lambda\in \mathcal{B}_{2t-1}(t)\star \mathcal{B}_{2t'-1}(t')$ as claimed.
	
	\smallskip
	
	\noindent\emph{Case (ii): if $\mathsf{x}$ lies in position $(1,t+1)$.} Then $\mu_1=t$ and $\lambda_1=\mu_1+\nu_1$ since $\mathsf{x}$ contains the last 1 of $\mathsf{F}$. Let $\mathsf{y}$ denote the box containing the last 2 in the reading order of $\mathsf{F}$; this exists as $\nu\ne(2t')$. 
	The box $\mathsf{y}$ must be leftmost in its row, as all of the 1s in $\mathsf{F}$ lie precisely in the first row of $[\lambda\setminus\mu]$. If $\mathsf{y}$ does not lie at the top of its column, it must lie immediately under a 1 in $\mathsf{F}$, from which we deduce that $\mathsf{y}$ occupies position $(2,j)$ for some $j\ge t+1$. 
	But then $\mu_2\ge t$, contradicting $|\mu|=2t-2$. Thus $\mathsf{y}$ lies at the top of its column and is an addable box for $\mu$. Moreover, $\mathsf{y}$ cannot lie in position $(t+1,1)$ or else $|\mu|\ge \mu_1+l(\mu)-1=t+t-1$. Thus, if $\mathsf{y}$ occupies position $(r,\mu_r+1)$ then $\hat{\mu}:=(\mu_1,\dotsc,\mu_{r-1},\mu_r+1,\mu_{r+1},\dotsc,\mu_{l(\mu)})\in\mathcal{B}_{2t-1}(t)$ (note $\hat{\mu}$ is well-defined since $\mu_r<\mu_{r-1}$). 
	
	Let $j\ge 2$ be such that $\nu_2=\nu_3=\dotsc=\nu_j>\nu_{j+1}$. Similarly to case (i), we define a Littlewood--Richardson filling $\mathsf{F'}$ of $[\lambda\setminus\hat{\mu}]$ to be obtained from $\mathsf{F}$ by removing the 2 corresponding to the box $\mathsf{y}$, then relabelling the last $i$ in $\mathsf{F}$ by the number $i-1$ for each $3\le i\le j$ (or no relabelling required if $j=2$). The resulting weight is $\hat{\nu}:=(\nu_1,\nu_2,\dotsc,\nu_{j-1},\nu_j-1,\nu_{j+1},\dotsc,\nu_{l(\nu)})\in\mathcal{B}_{2t'-1}(t'+1)$. Since $\lambda_1=\mu_1+\nu_1\le t+t'$, we must have $l(\nu)\le\nu_1\le t'$, 
	and so in fact $\hat{\nu}\in\mathcal{B}_{2t'-1}(t')$. 
	
	Thus setting $\mathsf{b}=\mathsf{y}$ and taking $\hat{\mu}$, $\hat{\nu}$ as described we find that $\lambda\in \mathcal{B}_{2t-1}(t)\star \mathcal{B}_{2t'-1}(t')$ as claimed.
	
	\smallskip
	
	\noindent\emph{Case (iii): if $\mathsf{x}$ lies in position $(t+1,1)$.} Let $\mathsf{z}$ denote the box containing the second-to-last 1 in the reading order of $\mathsf{F}$; this exists as $\nu\ne(1^{2t'})$. It cannot be in position $(t+1,2)$, or else $|\mu|\ge \mu'_1+\mu'_2=2t$. Thus $\mathsf{z}$ must be leftmost in its row (in some row $r<t$) and lie at the top of its column, 
	so it must be an addable box for $\mu$. Moreover, $\mathsf{z}$ cannot be in position $(1,t+1)$ as $|\mu|=2t-2$ 
	and so $\hat{\mu}:=(\mu_1,\dotsc,\mu_{r-1},\mu_r+1,\mu_{r+1},\dotsc,\mu_{l(\mu)})\in\mathcal{B}_{2t-1}(t)$. 
	
	Recall $\nu_1=\dotsc=\nu_k>\nu_{k+1}$. If $k=1$, then the filling $\mathsf{F'}$ defined as $\mathsf{F}$ restricted to the boxes of $[\lambda\setminus\hat{\mu}]$ is a Littlewood--Richardson filling of weight $\hat{\nu}:=(\nu_1-1,\nu_2,\dotsc,\nu_{l(\nu)})\in\mathcal{B}_{2t'-1}(t')$. 
	If $k>1$, then since the last 1 lies in the box $\mathsf{x}$ at position $(t+1,1)$, the last $i$ lies in position $(t+i,1)$ for each $2\le i\le k$, and notice that $\mu_2'\le t-2$ since $l(\mu)=t$. 
	Similarly to case (i), we define a Littlewood--Richardson filling $\mathsf{F'}$ of $[\lambda\setminus\hat{\mu}]$ to be obtained from $\mathsf{F}$ by removing the 1 corresponding to the box $\mathsf{z}$, then relabelling the second-to-last $i$ in $\mathsf{F}$ by the number $i-1$, for each $2\le i\le k$. The resulting weight is $\hat{\nu}:=(\nu_1,\dotsc,\nu_{k-1},\nu_k-1,\nu_{k+1},\dotsc,\nu_{l(\nu)})\in\mathcal{B}_{2t'-1}(t')$. 
	
	Thus setting $\mathsf{b}=\mathsf{z}$ and taking $\hat{\mu}$, $\hat{\nu}$ as described we find that $\lambda\in \mathcal{B}_{2t-1}(t)\star \mathcal{B}_{2t'-1}(t')$ as claimed.
\end{proof}

\begin{prop}\label{prop: B star B} 
	Let $n,n',t,t'\in\mathbb{N}$ be such that $\tfrac{n}{2}<t\le n$ and $\tfrac{n'}{2}<t'\le n'$. Then 
	$$\mathcal{B}_n(t) \star\mathcal{B}_{n'}(t') = \mathcal{B}_{n+n'}(t+t').$$
\end{prop}

\begin{proof}
	That $\mathcal{B}_n(t) \star\mathcal{B}_{n'}(t') \subseteq \mathcal{B}_{n+n'}(t+t')$ follows from Definition~\ref{def:star operator} and Lemma~\ref{lem: LRfirstpart}. 
	For the reverse inclusion, and hence equality of sets, we proceed by induction on the quantity $2t-n+2t'-n'\ge 2$, with the base case given by Lemma~\ref{lem: LR base case}. 
	Now suppose $2t-n+2t'-n'>2$, so without loss of generality assume $t'-1>\tfrac{n'-1}{2}$. 
	Then $\mathcal{B}_{n'-1}(t'-1)\star\mathcal{B}_1(1)=\mathcal{B}_{n'}(t')$ 
	and $\mathcal{B}_n(t)\star\mathcal{B}_{n'-1}(t'-1)=\mathcal{B}_{n+n'-1}(t+t'-1)$ by the inductive hypothesis. Thus 
	\begin{align*}
	\mathcal{B}_n(t)\star\mathcal{B}_{n'}(t') &= \mathcal{B}_n(t) \star\big( \mathcal{B}_{n'-1}(t'-1)\star\mathcal{B}_1(1) \big)\\ 
	&=\big( \mathcal{B}_n(t) \star\mathcal{B}_{n'-1}(t'-1) \big) \star\mathcal{B}_1(1)\\ 
	&=\mathcal{B}_{n+n'-1}(t+t'-1)\star\mathcal{B}_1(1)\\ 
	&=\mathcal{B}_{n+n'}(t+t') 
	\end{align*}
	as claimed.
\end{proof}

We remark that the hypotheses on $n,t,n'$ and $t'$ in Proposition~\ref{prop: B star B} are necessary. For instance, considering $(4,4)\in\mathcal{P}(8)$ we observe that $\mathcal{B}_7(3)\star\mathcal{B}_{1}(1)\neq\mathcal{B}_{8}(4)$.

\begin{lemma}\label{lem: 10.3}
	Let $n,m,t\in\mathbb{N}$. Suppose that $\tfrac{m}{2}<t\le m$ and that $n\ge 5$. Then
	$$\mathcal{B}_m(t) \star (\mathcal{B}_n(n-2)\cup \{(n)\}^\circ) = \mathcal{B}_{m+n}(t+n).$$
	In particular, $\mathcal{P}(m+n) = \mathcal{P}(m)\star(\mathcal{P}(n)\setminus\{(n-1,1)\}^\circ)$.
\end{lemma}

\begin{proof}
	If $t=1$ then $m=1$ and the result follows from the branching rule for symmetric groups \cite[9.2]{J}, so from now on we may assume $t\ge 2$.

	Let $X:=\mathcal{B}_m(t)\star(\mathcal{B}_n(n-2)\cup \{(n)\}^\circ)$. Since $n\ge 5$, we have that $n-2>\tfrac{n}{2}$, and so $\mathcal{B}_{m+n}(t+n-2)\subseteq X$ by Proposition~\ref{prop: B star B}. Moreover, $X\subseteq\mathcal{B}_m(t)\star\mathcal{P}(n)=\mathcal{B}_{m+n}(t+n)$, again by Proposition~\ref{prop: B star B}. Since $X^\circ=X$, it remains to show that if $\lambda\vdash m+n$ with $\lambda_1\in\{t+n-1,t+n\}$, then $\lambda\in X$.
	
	First suppose $\lambda_1=t+n$, so $\lambda=(t+n,\mu)$ for some $\mu\vdash m-t<t$. Observe that $c\6\lambda_{(t,\mu),(n)}\neq 0$ and that $(t,\mu)\in\mathcal{B}_m(t)$. It follows that $\lambda\in X$.
	
	Otherwise we have $\lambda_1=t+n-1$, so $\lambda=(t+n-1,\mu)$ for some $\mu\vdash m-t+1$. If $\mu_1\ge t$, then $m=2t-1$ and thus $\lambda=(t+n-1,t)$. Since $c\6\lambda_{(t,t-1),(n)}\neq 0$ and $(t,t-1)\in\mathcal{B}_m(t)$, then $\lambda\in X$. 
	If $l(\mu)\ge t$ then $m=2t-1$, $\lambda=(t+n-1,1^t)$, and we similarly conclude that $\lambda\in X$ since $(t,1^{t-1})\in\mathcal{B}_m(t)$.
	Otherwise, $\mu\in\mathcal{B}_{m-t+1}(t-1)$, so $(t-1,\mu)\in\mathcal{B}_m(t)$. Since $c\6\lambda_{(t-1,\mu),(n)}\neq 0$ we deduce that $\lambda\in X$.
\end{proof}

Recall from Section~\ref{sec: LR prelims} that $\mathcal{LR}(\gamma)$ denotes the set of weights of Littlewood--Richardson fillings of a skew shape $\gamma$, and from Section~\ref{sec: sn prelims} that $\nu^+$ denotes the set of partitions indexing the irreducible constituents of the induced character $\chi^\nu\up^{\fS_{|\nu|+1}}$, for any partition $\nu$.

\begin{lemma}\label{lem: skew filling}
	Let $X=[\lambda\setminus\mu]$ be a skew shape, and suppose $\nu\in\mathcal{LR}(X)$. 
	Let $Y$ be a skew shape obtained from $X$ by adding a single box. Then $\mathcal{LR}(Y)\cap \nu^+\ne\emptyset$.
\end{lemma}

\begin{proof}
	Let $|\lambda|=n$ and $|\mu|=m$.
	First suppose $Y$ is obtained from $X$ by adding a box externally, that is, $Y=[\tilde{\lambda}\setminus\mu]$ for some $\tilde{\lambda}\in\lambda^+$. Since $\nu\in\mathcal{LR}(X)$, the iterated Littlewood--Richardson coefficient $c^{\tilde{\lambda}}_{\mu, \nu, (1)} = \langle \chi^{\tilde{\lambda}}, \chi^\mu\times\chi^\nu\times\chi^{(1)}\up^{\fS_{n+m+1}}_{\fS_n\times\fS_m\times\fS_1}\rangle$ is positive. But by the branching rule,
	\begin{align*}
	0<c^{\tilde{\lambda}}_{\mu, \nu, (1)} &= \left\langle \chi^{\tilde{\lambda}}, \chi^\mu\times(\chi^\nu\times\chi^{(1)})\up^{\fS_n\times\fS_{m+1}}_{\fS_n\times\fS_m\times\fS_1} \up^{\fS_{n+m+1}}_{\fS_n\times\fS_{m+1}}\right\rangle \\
	&= \sum_{\tilde{\nu}\in\nu^+} \left\langle \chi^{\tilde{\lambda}}, \chi^\mu\times\chi^{\tilde{\nu}} \up^{\fS_{n+m+1}}_{\fS_n\times\fS_{m+1}} \right\rangle
	\end{align*}
	so there exists $\tilde{\nu}\in\nu^+$ such that $\left\langle \chi^{\tilde{\lambda}}, \chi^\mu\times\chi^{\tilde{\nu}} \up^{\fS_{n+m+1}}_{\fS_n\times\fS_{m+1}} \right\rangle>0$. Hence $\tilde{\nu}\in\mathcal{LR}(Y)$.
	
	Otherwise, $Y$ is obtained from $X$ by adding a box internally, that is, $Y=[\lambda\setminus\tilde{\mu}]$ for some $\tilde{\mu}\in\mu^-$. A similar argument 
shows that $\mathcal{LR}(Y)\cap \nu^+\ne\emptyset$ in this case also. 
\end{proof}

The following definition is crucial for our arguments in the next section.

\begin{defn}\label{def: Dset newOmegas}
	Let $q,m\in\mathbb{N}$ be such that $q\geq 2$ and let $\mathcal{B}\subseteq\mathcal{P}(m)$.
	Let $H=(\fS_{m})^{\times q}\leq \fS_{qm}$.
	We let $\mathcal{D}(q, m, \mathcal{B})$ be the subset of $\mathcal{P}(qm)$ consisting of all those partitions $\lambda\in\mathcal{P}(qm)$ for which there exists $\mu_1,\mu_2,\ldots, \mu_{q}\in\mathcal{B}$, not all equal, such that 
	$$\chi^{\mu_1}\times\chi^{\mu_2}\times\cdots\times\chi^{\mu_q}\ \Big{|}\ \chi^\lambda\down_{H}.$$
Notice that if $\mathcal{B}^\circ=\mathcal{B}$, then $\mathcal{D}(q,m,\mathcal{B})^\circ = \mathcal{D}(q,m,\mathcal{B})$.
\end{defn}

We now aim to show that under appropriate restrictions on the parameters we have that $\mathcal{B}_{qm}(qt-1) \subseteq \mathcal{D}\big(q,m,\mathcal{B}_m(t)\big).$ This is done in Proposition~\ref{prop: BinD} using an inductive argument. Lemma~\ref{prop: general m, q=2} below provides an important step towards the proof of Proposition \ref{prop: BinD}.

\begin{lemma}\label{prop: general m, q=2}
	Let $m,t\in\mathbb{N}$ and suppose $\tfrac{m}{2}+1<t\le m$. 
	Let $\lambda\in \mathcal{B}_{2m}(2t-1)$. Then either $\lambda\in \mathcal{D}\big(2,m,\mathcal{B}_m(t)\big)$, or $\chi^\lambda\down_{\fS_m\times\fS_m}$ has two irreducible constituents $\chi^\alpha\times\chi^\alpha$ and $\chi^\beta\times\chi^\beta$ where $\alpha\ne\beta\in\mathcal{B}_m(t)$.
\end{lemma}

\begin{proof}
	First suppose $\lambda=(m,m)$. Notice that $\chi^\alpha\times\chi^\beta$ is a constituent of $\chi^\lambda\down_{\fS_m\times\fS_m}$ with $\alpha\in\mathcal{B}_m(t)$ if and only if $\alpha=\beta=(\alpha_1,m-\alpha_1)$ with $\tfrac{m}{2}\le\alpha_1\le t$. Since $t\ge 2$ and $t>\tfrac{m}{2}+1$, we have that $|[\tfrac{m}{2},t]\cap\mathbb{N}|\geq 2$. Thus we can always find two irreducible constituents $\chi^\alpha\times\chi^\alpha$ and $\chi^\beta\times\chi^\beta$ where $\alpha\ne\beta\in\mathcal{B}_m(t)$ as required.
Arguing similarly we conclude that the same happens for $\lambda=(2\6m)=(m,m)'$.
	
	Now let $\lambda\in\mathcal{B}_{2m}(2t-1)\setminus\{(m,m)\}^\circ$. By Proposition~\ref{prop: B star B}, 
	there exist partitions $\mu\in\mathcal{B}_m(t)$ and $\nu\in\mathcal{B}_m(t-1)$ such that $c^\lambda_{\mu\nu}>0$. 
	If $\mu\ne\nu$ then $\lambda\in\mathcal{D}\big(2,2m,\mathcal{B}_m(t)\big)$ and we are done, so assume that $\mu=\nu\in\mathcal{B}_m(t-1)$. By `passing a box' between $[\mu]$ and $[\lambda\setminus\mu]$, we construct partitions
	\begin{itemize}\setlength\itemsep{0.5em}
		\item[(i)] $\hat{\mu}\in\mathcal{B}_{m+1}(t)$ and $\hat{\nu}\in\mathcal{B}_{m-1}(t-1)$ such that $c^\lambda_{\hat{\mu}\hat{\nu}}>0$; then
		\item[(ii)] $\tilde{\mu}\in\mathcal{B}_m(t)$ and $\tilde{\nu}\in\mathcal{B}_m(t)$ such that $c^\lambda_{\tilde{\mu}\tilde{\nu}}>0$, \emph{and} $\tilde{\mu}\ne\mu$, 
	\end{itemize}
	whence the assertion of the proposition follows (either since $\tilde{\mu}\ne\tilde{\nu}$ so $\lambda\in\mathcal{D}\big(2,m,\mathcal{B}_m(t)\big)$ directly, or $\tilde{\mu}=\tilde{\nu}$ but $\tilde{\mu}\ne\mu$). We now explain in detail the constructions (i) and (ii).
	
	\smallskip
	
	\noindent\textit{Step (i):} Fix a Littlewood--Richardson filling $\mathsf{F}$ of $[\lambda\setminus\mu]$ of weight $\nu$. Let $\mathsf{b}$ denote the box of $[\lambda\setminus\mu]$ containing the last 1 in the reading order of $\mathsf{F}$; clearly this is an addable box for $\mu$. We split into three cases depending on the shape of $[\mu]\cup\mathsf{b}$.
	
	\textit{Case (a):} If $[\mu]\cup\mathsf{b}$ is not a rectangle, then define $\hat{\mu}$ via $[\hat{\mu}]:=[\mu]\cup\mathsf{b}$. 
	Let $k\in\mathbb{N}$ be such that $\nu_1=\nu_2=\cdots=\nu_k>\nu_{k+1}$.
	Define $\mathsf{F'}$ to be obtained from $\mathsf{F}$ by removing the 1 corresponding to the box $\mathsf{b}$, and then if $k>1$ additionally relabelling the last $i$ in $\mathsf{F}$ by the number $i-1$, for each $2\le i\le k$. Thus $\mathsf{F}'$ is a Littlewood--Richardson filling of $[\lambda\setminus\hat{\mu}]$ of weight $\hat{\nu}:=(\nu_1,\dotsc,\nu_{k-1},\nu_k-1,\nu_{k+1},\dotsc,\nu_{l(\nu)})$, by the same argument as in the proof of Lemma~\ref{lem: LR base case}.
	
	Now we may assume $[\mu]\cup\mathsf{b}$ is a rectangle. Notice $m\ge 3$, so either $[(2,1)]\subseteq [\mu]$ or $\mu\in\{(m),(1^m)\}$. If $\mu=(m)$, then $\mathsf{F}$ being a filling of weight $\nu=\mu$ and the definition of $\mathsf{b}$ together imply that $\lambda=(2m)\notin \mathcal{B}_{2m}(2t-1)$, a contradiction. Similarly if $\mu=(1^m)$ then $\lambda=(1^{2m})\notin\mathcal{B}_{2m}(2t-1)$. Thus when $[\mu]\cup\mathsf{b}$ is a rectangle then $[(2,1)]\subseteq [\mu]$.
	
	\textit{Case (b):} If $[\mu]\cup\mathsf{b}$ is a rectangle and $l(\lambda)>l(\mu)$, let $\mathsf{c}$ be the box in row $l(\mu)+1$, column 1, and define $[\hat{\mu}]:=[\mu]\cup\mathsf{c}$. Suppose in $\mathsf{F}$ the box $\mathsf{c}$ is filled with the number $j$. Since the rows of $[\lambda\setminus\mu]$ are filled weakly increasingly, and the columns strictly increasingly, the $j$ in $\mathsf{c}$ must be the last $j$ that appears in the reading order of $\mathsf{F}$. Suppose $\nu_j=\nu_{j+1}=\dotsc=\nu_l>\nu_{l+1}$. Define $\mathsf{F'}$ to be obtained from $\mathsf{F}$ by removing the $j$ corresponding to the box $\mathsf{c}$, and then if $l>j$ additionally relabelling the last $i$ in $\mathsf{F}$ by the number $i-1$, for each $j+1\le i\le l$. Thus $\mathsf{F}'$ is a Littlewood--Richardson filling of $[\lambda\setminus\hat{\mu}]$ of weight $\hat{\nu}:=(\nu_1,\dotsc,\nu_{l-1},\nu_l-1,\nu_{l+1},\dotsc,\nu_{l(\nu)})$, by the same argument as in the proof of Lemma~\ref{lem: LR base case}.
	
	\textit{Case (c):} Otherwise $[\mu]\cup\mathsf{b}$ is a rectangle and $l(\lambda)=l(\mu)$. 
	If $l(\mu)>2$, then the number 2 appears in $\mathsf{F}$ precisely as the entries in the second row of $[\lambda\setminus\mu]$, and thus $\nu_2=\lambda_2-\mu_2$. 
	But the number $\nu_1$ of 1s in $\mathsf{F}$ is equal to $\lambda_1-\mu_1+1$ (they appear in the first row of $[\lambda\setminus\mu]$ and $\mathsf{b}$). Thus $\mu=\nu$ and $\mu_1=\mu_2$ give $\lambda_2=\lambda_1+1$, a contradiction. Thus $l(\mu)=2$, in which case $\mu$ is of the form $(a,a-1)\vdash m$, but since $l(\lambda)=l(\mu)=2$ and $\mu=\nu$ then in fact $\lambda=(m,m)$, a contradiction. Thus case (c) in fact cannot occur.
	
	Observe that in cases (a) and (b), $[\hat{\mu}]$ is obtained from $[\mu]$ by adding a single addable box, so $\mu\in\mathcal{B}_m(t-1)$ implies $\hat{\mu}\in\mathcal{B}_{m+1}(t)$. Also since $\nu\in\mathcal{B}_m(t-1)$, clearly $\hat{\nu}\in\mathcal{B}_{m-1}(t-1)$. 
	
	\smallskip
	
	\noindent\textit{Step (ii):} Let $\mathsf{x}=[\hat{\mu}]\setminus[\mu]$. By construction $[\hat{\mu}]$ is not a rectangle and hence $\mathsf{x}$ is not the only removable box of $[\hat{\mu}]$. 
	Choose a removable box of $\hat{\mu}$ different from $\mathsf{x}$, say $\mathsf{y}$. Let $\tilde{\mu}$ be defined via $[\tilde{\mu}]:=[\hat{\mu}]-\mathsf{y}$, so $\tilde{\mu}\ne\mu$. Also $\hat{\mu}\in\mathcal{B}_{m+1}(t)$, so $\tilde{\mu}\in\mathcal{B}_m(t)$. By Lemma~\ref{lem: skew filling}, there exists a Littlewood--Richardson filling of $[\lambda\setminus\hat{\mu}]\cup\mathsf{y}$ of weight $\tilde{\nu}$, for some $\tilde{\nu}\in\hat{\nu}^+$. 
	But $\hat{\nu}\in\mathcal{B}_{m-1}(t-1)$, so $\tilde{\nu}\in\mathcal{B}_m(t)$. 
\end{proof}

\begin{prop}\label{prop: BinD} 
	Let $m,t\in\mathbb{N}$ be such that $\tfrac{m}{2}+1<t\le m$. Let $q\in\mathbb{N}_{\ge 3}$. Then
	$$\mathcal{B}_{qm}(qt-1) \subseteq \mathcal{D}\big(q,m,\mathcal{B}_m(t)\big).$$
\end{prop}

\begin{proof}
	We proceed by induction on $q$, beginning with the base case $q=3$. Let $\lambda\in\mathcal{B}_{3m}(3t-1)$. Then $\mathcal{B}_{2m}(2t-1)\star\mathcal{B}_m(t) = \mathcal{B}_{3m}(3t-1)$ by Proposition~\ref{prop: B star B}, 
	and so there exists $\mu\in\mathcal{B}_{2m}(2t-1)$ and $\nu\in\mathcal{B}_m(t)$ such that $c^\lambda_{\mu\nu}>0$. By Lemma~\ref{prop: general m, q=2}, one of the following holds:
	\begin{itemize}\setlength\itemsep{0.5em}
		\item[(i)] $\mu\in\mathcal{D}\big(2,m,\mathcal{B}_m(t)\big)$, in which case $c^\mu_{\sigma\tau}>0$ for some $\sigma\ne\tau\in\mathcal{B}_m(t)$. Then $c^\lambda_{\sigma\tau\nu}>0$ and hence $\lambda\in\mathcal{D}\big(3,m,\mathcal{B}_m(t)\big)$; or
		\item[(ii)] there exist two distinct partitions $\alpha, \beta\in\mathcal{B}_m(t)$ such that $c^\mu_{\alpha\alpha}\neq 0\neq c^\mu_{\beta\beta}$.
Then $c^\lambda_{\alpha\alpha\nu}\neq 0\neq c^\lambda_{\beta\beta\nu}$. Since $\nu\ne\alpha$ or $\nu\ne\beta$, we deduce that $\lambda\in\mathcal{D}\big(3,m,\mathcal{B}_m(t)\big)$ in this case as well.
	\end{itemize}
	
	Now suppose $q\ge 4$ and assume the statement of the proposition holds for $q-1$. Let $\lambda\in\mathcal{B}_{qm}(qt-1)$. Then there exists $\mu\in\mathcal{B}_{(q-1)m}\big((q-1)t-1\big)$ and $\nu\in\mathcal{B}_m(t)$ such that $c^\lambda_{\mu\nu}>0$, by Proposition~\ref{prop: B star B}. 
	By the inductive hypothesis, $\mu\in\mathcal{D}\big(q-1,m,\mathcal{B}_m(t)\big)$, so there exists $\mu_1,\dotsc,\mu_{q-1}\in\mathcal{B}_m(t)$ which are not all equal such that $c^{\mu}_{\mu_1\dotsc\mu_{q-1}}>0$. Hence $c^\lambda_{\mu_1\dotsc\mu_{q-1}\nu}>0$, which gives $\lambda\in \mathcal{D}\big(q,m,\mathcal{B}_m(t)\big)$.
\end{proof}

\smallskip


\section{The prime power case}\label{sec: new omegas prime power}

The fundamental part of Theorem~B  is the case when $n$ is a power of $p$, which we address in this section.
Let $k\in\mathbb{N}$ and let $\phi\in\Lin(P_{p^k})\setminus\{\triv_{P_{p^k}}\}$. 
The aim of this section is to determine the following numbers:
\[ m(\phi) = \max\{x\in\mathbb{N} \mid \mathcal{B}_{p^k}(x) \subseteq\Omega(\phi)\}\quad\mathrm{and}\quad M(\phi) = \min\{x\in\mathbb{N} \mid \Omega(\phi)\subseteq\mathcal{B}_{p^k}(x)\}.\]

Recall that if $\phi\in\mathrm{Lin}(P_{p^k})$ corresponds to $s\in [\overline{p}]^k$ (see Section 2.2) then we sometimes denote $m(\phi)$ and $M(\phi)$ by $m(s)$ and $M(s)$ respectively. 

Let $\phi=\phi(s)$ for some $s=(s_1,\dotsc,s_k)\in[\overline{p}]^k$. Since $\phi\ne\triv_{P_{p^k}}$ then $s\ne (0,\dotsc,0)$ and hence $f(s)$ is well-defined. We denote by $s^-$ the sequence $(s_1,s_2,\dotsc,s_{k-1})\in[\overline{p}]^{k-1}$. 
The main strategy is to induct on $k$, computing $M(s)$ and $m(s)$ from $M(s^-)$ and $m(s^-)$ respectively. 
(This is achieved in Theorems~\ref{thm: prime power M} and~\ref{thm:prime power m} below.)
With this in mind, a first key step is to investigate the relationship between the sets $\Omega(s^-)$ and $\Omega(s)$, for $\phi(s)\in\Lin(P_{p^k})$ and $k\in\mathbb{N}_{\ge 2}$. This is done in the following lemma. 

\begin{lemma}\label{lem: DinOm} 
	Let $p$ be an odd prime and $k\in\mathbb{N}_{\ge 2}$. Let $s=(s_1,\dotsc,s_k)\in[\overline{p}]^k$ and write $s^-=(s_1,\ldots, s_{k-1})$. Then 
	$$\mathcal{D}(p,p^{k-1},\Omega(s^-))\subseteq \Omega(s).$$
\end{lemma}

\begin{proof}
	We consider the following subgroups of $\fS_{p^k}$: let $P=P_{p^k}=P_{p^{k-1}}\wr P_p$, let $B=(P_{p^{k-1}})^{\times p}$ be the base group of the wreath product $P$, and let $H=(\fS_{p^{k-1}})^{\times p}\le \fS_{p^k}$, naturally containing $B$. Also let $W=PH=H\rtimes P_p\le\fS_{p^k}$, so $W\cong\fS_{p^{k-1}}\wr P_p$. 
	
	Let $\lambda\in\mathcal{D}(p,p^{k-1},\Omega(s^-))$, so $\chi^\lambda\down_H$ has a constituent $\psi:=\chi^{\mu_1}\times\cdots\times\chi^{\mu_p}\in\Irr(H)$ such that $\mu_1\ldots, \mu_p\in\Omega(s^-)$ are not all equal. Since $\chi^\lambda\in\Irr(\fS_{p^k}\mid\psi)$, there exists $\chi\in\Irr(W\mid\psi)$ such that $\chi$ is a constituent of $\chi^\lambda\down_W$. Since $\mu_1,\dotsc,\mu_p$ are not all equal, then $\chi=\psi\up^W_H$ by the description of $\Irr(\fS_{p^k}\wr P_p)$ given in Section~\ref{sec:wreath}. Since $PH=W$ and $P\cap H=B$, we have that $\chi\down_P=\psi\down_B\up^P$ by \cite[Problem 5.2]{IBook}. Moreover, $\psi\down_B = \chi^{\mu_1}\down_{P_{p^{k-1}}}\times\cdots\times\chi^{\mu_p}\down_{P_{p^{k-1}}}$, so $\phi(s^-)^{\times p}$ is a constituent of $\psi\down_B$. Thus $\phi(s^-)^{\times p}\up^P$ is a direct summand of $\chi^\lambda\down_P$. By Lemma~\ref{lem: easy observation}, we deduce that $\phi(s)=\mathcal{X}(\phi(s^-);\phi_{s_k})$ is a constituent of $\chi^\lambda\down_P$. Thus $\lambda\in\Omega(s)$, as desired.
\end{proof}

Given a sequence $s$ and its shorter subsequence $s^-$, 
Lemma~\ref{lem: DinOm} allows us to deduce information about $\Omega(s)$ from the knowledge of $\Omega(s^-)$. In order to exploit this, we need to understand the structure of $\Omega(t)$ for short or `minimal' sequences $t$. This is done in the next two lemmas.

\begin{lemma}\label{lem:n=1} 
	Let $p$ be an odd prime and let $x\in[\overline{p}]$. Then 
	$$\Omega(x) = \begin{cases}
	\mathcal{P}(p)\setminus\{(p-1,1),(2,1^{p-2})\}&\mathrm{if}\ x=0,\\
	\mathcal{P}(p)\setminus\{(p),(1^p) \} = \mathcal{B}_p(p-1) &\mathrm{if}\ x\in[p-1].
	\end{cases}$$
\end{lemma}
\begin{proof}
This is a direct consequence of Lemma~\ref{lem: MN}. 
\end{proof}

\begin{lemma}\label{lem: 001}
	Let $p$ be an odd prime, $k\in\mathbb{N}_0$ and $s=(0,\dotsc,0,x)\in [\overline{p}]^{k+1}$ where $x\ne 0$. Then $\Omega(s)=\mathcal{B}_{p^{k+1}}(p^{k+1}-1)$, except if $(p,k)=(3,1)$, in which case $\Omega\big((0,1)\big)=\mathcal{B}_9(8)\setminus\{(3^3)\}$. Moreover, for all such $p$, $k$ and $s$, $\big\langle \chi^{(p^{k+1}-1,1)}\down_{P_{p^{k+1}}}, \phi(s)\big\rangle = 1$.
\end{lemma}

\begin{proof}
	The assertion can be checked directly if $(p,k)=(3,1)$ and follows from Lemma~\ref{lem:n=1} when $k=0$. Now assume $k\ge 2$ if $p=3$, or $k\ge 1$ if $p\ge 5$. To ease the notation we set $P:=P_{p^{k+1}}$ and $\mathcal{D}:=\mathcal{D}(p,p^k, \Omega(s^-))$ for the rest of the proof. Since $s^-=(0,\ldots, 0)$, by \cite[Proposition 3.8]{GL1} we know that $\mathcal{D}=\mathcal{B}_{p^{k+1}}(p^{k+1}-2)$. Using Lemma~\ref{lem: DinOm}, we deduce that $\mathcal{B}_{p^{k+1}}(p^{k+1}-2)\subseteq\Omega(s)$.
	
	Since $\Omega(s)$ is closed under conjugation, in order to conclude that $\Omega(s)=\mathcal{B}_{p^{k+1}}(p^{k+1}-1)$ it remains to show that $(p^{k+1})\notin\Omega(s)$ and that $(p^{k+1}-1,1)\in\Omega(s)$. 
The first assertion is obvious as $\chi^{(p\6{k+1})}\down_P=\triv_P\ne\phi(s)$. 
On the other hand, if $\lambda=(p^{k+1}-1,1)$, then \cite[Corollary 9.1]{PW} implies that $\mathcal{X}\big((p^k);(p-1,1)\big)$ is an irreducible constituent of $\chi^\lambda\down^{\fS_{p^{k+1}}}_{\fS_{p^k}\wr\fS_p}$ appearing with multiplicity $1$. Moreover,
	$$\mathcal{X}\big((p^k); (p-1,1)\big)\down^{\fS_{p^k}\wr\fS_p}_P = \mathcal{X}\left(\chi^{(p^k)}\down_{P_{p^k}}^{\fS_{p^k}}; \chi^{(p-1,1)}\down_{P_p}^{\fS_p}\right) = \sum_{z=1}^{p-1} \mathcal{X}(\triv_{P_{p^k}}; \phi_z).$$ 
	This shows that $\phi(s)$ is an irreducible constituent of $\chi^\lambda\down_P$. Hence $(p^{k+1}-1,1)\in\Omega(s)$.
	
	Keeping $\lambda=(p^{k+1}-1,1)$, we now wish to show that $\langle\chi^\lambda\down_P,\phi(s)\rangle=1$. Let $H\cong(\fS_{p^k})^{\times p}$ and $B\cong (P_{p^k})^{\times p}$ with $B\le H$. From \cite[Lemma 3.2]{GTT} and the Littlewood--Richardson rule we see that $\chi^\lambda\down_H = (p-1)\triv_H+\Theta$, where 
	$$\Theta=(\chi^\mu\times\triv\times\cdots\times\triv)+(\triv\times\chi^\mu\times\cdots\times\triv)+\cdots +(\triv\times\cdots\times\triv\times\chi^\mu)\ \text{and}\ \mu=(p^k-1,1).$$
	(Here $\triv$ denotes $\triv_{\fS_{p^k}}$.) Since we already know that $\mathcal{X}\big((p^k); (p-1,1)\big)$ is a constituent of $\chi^\lambda\down_{\fS_{p^k}\wr \fS_p}$, then by the description of irreducible characters of wreath products \cite[\textsection 4.3]{JK},
	$$\chi^\lambda\down_{\fS_{p^k}\wr P_p} = \sum_{i=1}^{p-1}\mathcal{X}(\triv_{\fS_{p^k}}; \phi_i) + (\underbrace{\chi^\mu\times\triv\times\cdots\times\triv}_{=:\alpha})\up^{\fS_{p^k}\wr P_p}_H.$$
	Finally, $\langle\mathcal{X}(\triv_{\fS_{p^k}};\phi_i)\down_P, \phi(s)\rangle = \langle\mathcal{X}(\phi(s^-),\phi_i), \mathcal{X}(\phi(s^-),\phi_x)\rangle = \delta_{ix}$, and
	$$\langle \alpha\up^{\fS_{p^k}\wr P_p}_H\down_P, \phi(s)\rangle = \langle \alpha\down_B\up^P, \phi(s)\rangle = \langle\alpha\down_B, \phi(s)\down_B\rangle 
	= \langle\chi^\mu\down_{P_{p^k}}, \triv_{P_{p^k}}\rangle = 0,$$
	where the first equality in the line above follows from \cite[Problem 5.2]{IBook} and the last one follows from \cite[Theorem A]{GL1}. Since $x\in[p-1]$, we deduce that $\langle\chi^\lambda\down_P,\phi(s)\rangle=1$.
\end{proof}

Before proceeding with the determination of $m(s)$ and $M(s)$, we encourage the reader to recall the notation introduced in Definition~\ref{def:f and eta}. It turns out that the positions $f(s)$ and $g(s)$ of the leading non-zeros in the sequence $s$ govern the form of $\Omega(s)$. 

\smallskip

First, we determine the value of $M(s)$. This is done in Proposition~\ref{prop: upper bound} and Theorem~\ref{thm: prime power M} below, for all odd primes.

\begin{prop}\label{prop: upper bound} 
	Let $p$ be an odd prime, $k\in\mathbb{N}$, and $s\in[\overline{p}]^k\setminus U_k(0)$. 
	Then $M(s)\le p^k-p^{k-f(s)}$.
\end{prop}

\begin{proof}
	We proceed by induction on $k-f(s)$. The base case $f(s)=k$ follows from Lemma~\ref{lem: 001}. 
	Now suppose that $k\geq 2$ and that $f(s)<k$. In this setting we have that $f(s)=f(s^-)$. Let $\lambda\notin\mathcal{B}_{p^k}(p^k-p^{k-f(s)})$, so we may without loss of generality assume 
	$$\lambda_1>p^k-p^{k-f(s)}=p(p^{k-1}-p^{k-1-f(s^-)}).$$ 
	Lemma~\ref{lem: LRfirstpart} implies that for each irreducible constituent $\chi^{\mu_1}\times\cdots\times\chi^{\mu_p}$ of $\chi^\lambda\down_{(\fS_{p^{k-1}})^{\times p}}$, there exists some $i\in[p]$ such that $(\mu_i)_1>p^{k-1}-p^{k-1-f(s^-)}$. 
	By the inductive hypothesis, $M(s^-)\leq p^{k-1}-p^{k-1-f(s^-)}$, and hence $\mu_i\notin \Omega(s^-)$ since $\mu_i\notin\mathcal{B}_{p^{k-1}}(p^{k-1}-p^{k-1-f(s^-)})$.
	
	Suppose that $\lambda\in\Omega(s)$, then $\phi(s)\down_{(P_{p^{k-1}})^{\times p}}=\phi(s^-)^{\times p}$ is a constituent of $\chi^\lambda\down_{(P_{p^{k-1}})^{\times p}}$. Since $\phi(s^-)^{\times p}$ is irreducible, it must be a constituent of $$\chi^{\mu_1}\down_{P_{p^{k-1}}}\times\cdots\times\chi^{\mu_p}\down_{P_{p^{k-1}}}$$ 
	for some $\chi^{\mu_1}\times\cdots\times\chi^{\mu_p}$ as described above. In particular, this implies that $\mu_1,\dotsc,\mu_p\in\Omega(s^-)$, a contradiction. Hence $\lambda\notin\Omega(s)$, and we conclude that $\Omega(s)\subseteq \mathcal{B}_{p^k}(p^k-p^{k-f(s)})$.
\end{proof}

\begin{thm}\label{thm: prime power M}
	Let $p$ be an odd prime, $k\in\mathbb{N}$, and $s\in[\overline{p}]^k\setminus U_k(0)$. 
	Then $M(s) = p^k-p^{k-f(s)}$.
\end{thm}

\begin{proof}
By Proposition~\ref{prop: upper bound} we know that $\Omega(s)\subseteq\mathcal{B}_{p^k}(p^k-p^{k-f(s)})$. Hence, it remains to exhibit a partition $\lambda\in\Omega(s)$ such that $\lambda_1=p^k-p^{k-f(s)}$. We proceed by induction on $k-f(s)$. For the base case $f(s)=k$, we have $\Omega(s)=\mathcal{B}_{p^k}(p^k-1)$ from Lemma~\ref{lem: 001}, which implies that $\lambda=(p^k-1,1)\in\Omega(s)$. 
	
Suppose now that $f(s)<k$. Then $f(s)=f(s^-)$ and there exists some partition $\mu=(\mu_1,\dotsc,\mu_m)\in\Omega(s^-)$ such that $\mu_1=p^{k-1}-p^{k-1-f(s^-)}$, by the inductive hypothesis. 
We distinguish two cases depending on the value of $s_k$. Let 	
$$\lambda = 
\begin{cases} 
(p\mu_1,p\mu_2,\dotsc,p\mu_{m-1},p(\mu_m-1)+p-1,1) & \mathrm{if}\ s_k\neq 0 , \\ 
\\
(p\mu_1,p\mu_2,\dotsc,p\mu_{m-1},p\mu_m) & \mathrm{if}\  s_k=0,\\ 
\end{cases}\ \ \ \text{and}\ \ \nu=\begin{cases} 
(p-1,1) & \mathrm{if}\ s_k\neq 0 , \\ 
\\
(p) & \mathrm{if}\  s_k=0.\\ 
\end{cases}$$
By \cite[Corollary 9.1]{PW}, $\mathcal{X}\big(\mu; \nu\big)$ is a constituent of $\chi^\lambda\down^{\fS_{p^k}}_{\fS_{p^{k-1}}\wr\fS_p}$, whence
$$\mathcal{X}\big(\mu;\nu\big)\down_{P_{p^k}}\ \Big|\ \chi^\lambda\down_{P_{p^k}}.$$
Since $\mu\in\Omega(s^-)$ by assumption, and since $\nu\in\Omega(s_k)$	by Lemma~\ref{lem:n=1}, we deduce that $$\phi(s)=\mathcal{X}(\phi(s^-);\phi_{s_k})\ \Big|\ \mathcal{X}\big(\chi^\mu\down^{\fS_{p^{k-1}}}_{P_{p^{k-1}}}; \chi^{\nu}\down^{\fS_p}_{P_p}\big)=\mathcal{X}(\mu; \nu)\down_{P_{p^k}}.$$
Thus $\phi(s)$ is an irreducible constituent of $\chi^\lambda\down_{P_{p^k}}$ and therefore $\lambda\in\Omega(s)$. Since $\lambda_1=p^k-p^{k-f(s)}$, the proof is concluded.
\end{proof}

In the proof of Theorem \ref{thm: prime power M}, for every sequence $s$ we exhibited a partition $\lambda\in\Omega(s)$ such that $\lambda_1=M(s)$. We can do much better in fact. In the following lemma we determine all partitions of $\Omega(s)$ having first part of maximal size (i.e. equal to $M(s)$). This will also be very important when proving Theorem A. 

\begin{lemma}\label{lem: 9.8}
	Let $p$ be an odd prime. Let $1\leq z\le k\in\mathbb{N}$, and let $s\in U_k(z)$ be such that $f(s)<k$. Then
	$$\Omega(s) \cap\{\lambda\vdash p^k \mid \lambda_1=M(s)\}^\circ = \{(M(s),\mu) \mid \mu\in\Omega(s_{f(s)+1},\dotsc,s_k) \}^\circ.$$
	In particular, if $z\ge 2$ then $\Omega(s) \cap\{\lambda\vdash p^k\ |\ \lambda_1=M(s)\}^\circ$ contains no thin partitions.
\end{lemma}

\begin{proof}
	Let $f=f(s)$, $t=(s_1,\ldots, s_f)$ and $u=(s_{f+1},\dotsc,s_k)$.
	Let $W=\fS_{p^f}\wr \fS_{p^{k-f}}\leq \fS_{p^k}$ and let $Y$ be the base group of the wreath product $W$, namely $Y=(\fS_{p^f})^{\times p^{k-f}}\leq W$. 
	Let $P=P_{p^k}=P_{p^f}\wr P_{p^{k-f}}\leq W$, and finally denote by $B$ the base group of $P$, that is, $B=(P_{p^f})^{\times p^{k-f}}\leq Y$.
	
	Let $\lambda=(M(s), \mu)\in\mathcal{P}(p^k)$ for some $\mu\in\mathcal{P}(p^{k-f})$. It suffices to prove the following two statements:
	\begin{itemize}\setlength\itemsep{0.5em}
		\item[(i)] $\left\langle\chi^{\lambda}\down_P, \phi(s)\right\rangle=\left\langle\mathcal{X}((p^f-1,1); \mu)\down_P, \phi(s)\right\rangle;$ and
		\item[(ii)] $\left\langle\mathcal{X}((p^f-1,1); \mu)\down_P, \phi(s)\right\rangle > 0$ if and only if $\mu\in\Omega(u)$.
	\end{itemize}
	The first assertion of the lemma then follows, since $\Omega(s)$ is closed under conjugation. The second statement follows simply from the observation that if $z\ge 2$ then $u\ne(0,\dotsc,0)\in[\overline{p}]^{k-f}$, and hence $\{(p^{k-f}), (1^{p^{k-f}})\}\cap\Omega(u)=\emptyset$.
	
	We now prove (i) and (ii). For convenience, let $\alpha=(p^f-1,1)$ and $q=p^{k-f}$.
	
	\smallskip
	
	\noindent\textbf{(i)} By Theorem~\ref{thm: prime power M}, $M(s)=p^k-p^{k-f}=q(p^f-1)$. Hence Lemma~\ref{lem: LRfirstpart} implies that given $\mu_1,\dotsc,\mu_q\vdash p^f$ such that $c^\lambda_{\mu_1,\dotsc,\mu_q}\neq 0$, then either $\mu_1=\cdots=\mu_q=\alpha$ or there exists $j\in[q]$ such that $\mu_j=(p^f)$. 
	Since $(p^f)\notin\Omega(t)$ by Lemma~\ref{lem: 001}, 
	it follows that
	$$\left\langle\chi^\lambda\down_B, \phi(t)^{\times q}\right\rangle=\left\langle\chi^\lambda\down_Y,(\chi^\alpha)^{\times q}\right\rangle\cdot\left\langle(\chi^\alpha)^{\times q}\down_B, \phi(t)^{\times q}\right\rangle .$$  
	Moreover, by Lemma~\ref{lem: iteratedLR} we have that
	$$\left\langle\chi^\lambda\down_Y,(\chi^\alpha)^{\times q}\right\rangle=c^\lambda_{\alpha,\ldots, \alpha} = c^\mu_{(1),\ldots, (1)} =\chi^\mu(1),$$ 
	and thus $\left\langle\chi^\lambda\down_B, \phi(t)^{\times q}\right\rangle = \chi^\mu(1)\cdot\big(\langle\chi^\alpha\down_{P_{p^f}},\phi(t)\rangle\big)^q$.
	By \cite[Corollary 9.1]{PW} we know that $\mathcal{X}(\alpha;\mu)$ is an irreducible constituent of $\chi^\lambda\down_W$. 
	Moreover,
	$$\left\langle\mathcal{X}(\alpha;\mu)\down_Y,(\chi^\alpha)^{\times q}\right\rangle=\chi^\mu(1).$$ 
	Writing $\chi^\lambda\down_W = \mathcal{X}(\alpha;\mu) + \Delta$ for some character $\Delta$ of $W$, and $\mathcal{X}(\alpha;\mu)\down_Y = \chi^\mu(1)\cdot(\chi^\alpha)^{\times q} + \theta$ for some character $\theta$ of $Y$, we have that
	\begin{align*}
	\langle\chi^\lambda\down_B, \phi(t)^{\times q}\rangle &= \langle\mathcal{X}(\alpha;\mu)\down^W_B, \phi(t)^{\times q}\rangle + \langle\Delta\down^W_B, \phi(t)^{\times q}\rangle\\
	&=\chi^\mu(1)\cdot\big(\langle\chi^\alpha\down_{P_{p^f}},\phi(t)\rangle\big)^q + \langle\theta\down^Y_B, \phi(t)^{\times q}\rangle + \langle\Delta\down^W_B, \phi(t)^{\times q}\rangle,
	\end{align*}
	and therefore 
	$$\langle\theta\down^Y_B, \phi(t)^{\times q}\rangle = \langle\Delta\down^W_B, \phi(t)^{\times q}\rangle =0.$$
	Letting $c = \langle\Delta\down^W_P, \phi(s)\rangle$, then since $\phi(s)\down^P_B = \phi(t)^{\times q}$, we have that
	$$0 = \langle\Delta\down^W_B, \phi(t)^{\times q}\rangle \ge c\langle\phi(s)\down^P_B, \phi(t)^{\times q}\rangle = c,$$
	from which we conclude $c=0$. Thus $\left\langle\chi^{\lambda}\down_P, \phi(s)\right\rangle=\left\langle\mathcal{X}(\alpha; \mu)\down_P, \phi(s)\right\rangle$.
	
	\smallskip
	
	\noindent\textbf{(ii)} Now let $\gamma=\chi^\alpha\down^{\fS_{p^f}}_{P_{p^f}}$. By Lemma~\ref{lem: 001}, $\langle\gamma,\phi(t)\rangle = 1$. 
	Moreover, we observe that
	$$\mathcal{X}(\alpha;\mu)\down_P = \mathcal{X}(\alpha;\mu)\down^{\fS_{p^f}\wr\fS_q}_{P_{p^f}\wr P_q} =  \mathcal{X}(\gamma;\chi^\mu\down^{\fS_q}_{P_q}) = \sum_{\tau\in\Irr(P_q)} \langle\chi^\mu\down_{P_q}, \tau\rangle \cdot\mathcal{X}(\gamma;\tau).$$
	Since $\phi(s)=\mathcal{X}(\phi(t);\phi(u))$, we have that
	\begin{align*}
	\langle\mathcal{X}(\alpha; \mu)\down_P, \phi(s)\rangle &= \sum_{\tau\in\Irr(P_q)} \langle\chi^\mu\down_{P_q}, \tau\rangle\cdot\langle \mathcal{X}(\gamma;\tau), \phi(s)\rangle\\
	&= \sum_{\tau\in\Irr(P_q)} \langle\chi^\mu\down_{P_q}, \tau\rangle\cdot\delta_{\phi(u),\tau} = \langle\chi^\mu\down_{P_q}, \phi(u)\rangle,
	\end{align*}
	where the second equality follows from Lemma~\ref{lem: 9.6}. By definition of $\Omega(u)$, $\langle\chi^\mu\down_{P_q}, \phi(u)\rangle>0$ if and only if $\mu\in\Omega(u)$. This concludes the proof. 
\end{proof}

Our next primary goal is to determine $m(s)$.
This is done in Theorem \ref{thm:prime power m} below. The main ingredients for proving this theorem are Lemma \ref{lem: 9.9} and Proposition \ref{prop:case1} below,
whose proofs are quite technical. To help the reader appreciate how to use these two results to prove Theorem \ref{thm:prime power m}, we have postponed their proofs to Section \ref{sec:proofs}.

\smallskip

Recall that thin partitions $\tworow{n}{m}$ and $\hook{n}{m}$ were introduced in Definition~\ref{def:thin}.

\begin{lemma}\label{lem: 9.9}
	Let $p\geq 5$ be a prime and $k\in\mathbb{N}$. Let $s=(s_1,\dotsc,s_k)\in U_k(1)$, $f=f(s)$ and $x\in[\overline{p}]$. Then 
	\begin{itemize}\setlength\itemsep{0.5em}
		\item[(a)] $\Omega(s,x) = \mathcal{B}_{p^{k+1}}(p^{k+1}-p^{k+1-f}-1)\sqcup \{(p^{k+1}-p^{k+1-f},\mu) : \mu\in\Omega(s_{f+1},\dotsc,s_k,x) \}^\circ$.
		\item[(b)] Moreover, if $x\ne 0$ and $\lambda\in\{ \tworow{p^{k+1}}{p^{k+1}-p^{k+1-f}-1}, \hook{p^{k+1}}{p^{k+1}-p^{k+1-f}-1} \}^\circ$, then $\langle\chi^\lambda\down_{P_{p^{k+1}}},\phi(s,x)\rangle\ge 2$.
	\end{itemize}
\end{lemma}

For convenience, we sometimes identify a partition $\lambda$ with its corresponding character $\chi^\lambda$.

\begin{prop}\label{prop:case1}
	Let $p\geq 5$ be a prime and $k\in\mathbb{N}$. Suppose that $s\in[\overline{p}]^k$ satisfies the following:
	\begin{itemize}\setlength\itemsep{0.5em}
		\item[(i)] $m(s)>\tfrac{p^k}{2}+1$ and $\Omega(s)\setminus\mathcal{B}_{p^k}(m(s))$ contains no thin partitions. 
		\item[(ii)] $\langle \tworow{p^k}{m(s)}\down_{P_{p^k}}, \phi(s)\rangle \ge 2$ and $\langle \hook{p^k}{m(s)}\down_{P_{p^k}}, \phi(s)\rangle \ge 2$.
	\end{itemize}
	Then, for all $x\in[\overline{p}]$, 
	$$m(s,x)=p\cdot m(s),$$ 
	$\Omega(s,x)\setminus\mathcal{B}_{p^{k+1}}(pm(s))$ contains no thin partitions, and 
	$$\langle \tworow{p^{k+1}}{pm(s)}\down_{P_{p^{k+1}}}, \phi(s,x)\rangle \ge 2\quad \text{and}\quad \langle \hook{p^{k+1}}{pm(s)}\down_{P_{p^{k+1}}}, \phi(s,x)\rangle \ge 2.$$
\end{prop}

Roughly speaking, Proposition~\ref{prop:case1} tells us that whenever the multiplicity of $\phi(s)$ as a constituent of the restriction of the irreducible characters labelled by the \textit{longest} thin partitions in $\Omega(s)$ is greater than or equal to $2$, then $m(s,x)=p\cdot m(s)$ for all $x\in [\overline{p}]$. 
The assumptions in Proposition~\ref{prop:case1} might seem artificial, but they in fact mimic exactly the structure of the sets $\Omega(s)$ (see Lemma~\ref{lem: 9.9}, for instance).

\smallskip

We can now determine the value of $m(s)$ for all $s\in[\overline{p}]^k$. In fact, we prove much more about the structure of $\Omega(s)$. Our main strategy is to compute $m(s)$ inductively. That is, knowing $m(s)$ and other particular features of $\Omega(s)$, we exploit Proposition \ref{prop:case1} to compute $m(s,x)$ for any $x\in[\overline{p}]$. 
This technique is illustrated concretely in Example \ref{example: first} below. 

Recall that $\tau(s)$ was introduced in Definition~\ref{def: m(s) types}.

\smallskip

\begin{thm}\label{thm:prime power m}
	Let $p\ge 5$ be a prime. Let $k\in\mathbb{N}$ and let $\phi(s)\in\Lin(P_{p^k})\setminus\{\triv_{P_{p^k}}\}$. Then
	\[ m(s) = \begin{cases}
	p^k-p^{k-f(s)} - 1 + \delta_{f(s),k} & \mathrm{if}\ \tau(s)\neq 4,\\
	p^k-p^{k-f(s)}-p^{k-g(s)} & \mathrm{if}\ \tau(s)=4. \end{cases} \]
Moreover, $\Omega(s)\setminus B_{p^k}(m(s))$ contains no thin partitions. 
\end{thm}

\begin{proof}
The assertion follows from Lemma~\ref{lem: 001} if $\tau(s)=3$, and from Lemma~\ref{lem: 9.9} if $\tau(s)=2$.

Now suppose that $\tau(s)=4$. By Lemmas~\ref{lem: 001} and~\ref{lem: 9.9}, $t:=(s_1,\dotsc,s_{g(s)})$ satisfies conditions (i) and (ii) of Proposition~\ref{prop:case1}. In particular we have that $m(t) = p^{g(s)}-p^{g(s)-f(s)}-1$. Applying Proposition \ref{prop:case1} to the sequence $t$, we deduce that $m(t,s_{g(s)+1})=p\cdot m(t)$ and that $(t,s_{g(s)+1})$ also satisfies both conditions of Proposition~\ref{prop:case1}. The statement now follows by repeated application of Proposition~\ref{prop:case1} ($k-g(s)$ iterations), giving $m(s) = p^{k-g(s)}\cdot m(t) = p^k-p^{k-f(s)}-p^{k-g(s)}$.
\end{proof}

\begin{rem}\label{rem:inductive step}
Proposition~\ref{prop:case1} in fact governs the entire inductive step from $m(s)$ to $m(s,x)$, with Lemmas~\ref{lem: 001} and~\ref{lem: 9.9} providing the base cases (we may ignore $\triv_{P_{p^k}}$, corresponding to $s=(0,\dotsc,0)$, which was considered in \cite{GL1}). This can be seen easily using Figure~\ref{fig:flowchart p>=5}. 

\smallskip

We remark that we could not directly apply Propositions~\ref{prop:case1} to sequences $s=(0,\dotsc,0,s_{f(s)})\in[\overline{p}]^k$, since condition (ii) on multiplicities is not satisfied (see Lemma \ref{lem: 001}). Furthermore, we also could not apply the proposition to sequences of the form $s=(0,\dotsc,0,s_{f(s)},0,\dotsc,0)$, since $\Omega(s)\setminus\mathcal{B}_{p^k}(m(s))$ \emph{does} contain thin partitions, namely $\tworow{p^k}{m(s)+1}$, $\hook{p^k}{m(s)+1}$ and their conjugates. This is why Lemma \ref{lem: 9.9} is necessary. 
\hfill$\lozenge$

\begin{figure}[ht]
	\centering
	\begin{tikzpicture}[scale=1.2, every node/.style={scale=0.9},thick]
	\draw (0,0) node[] (2) {Lem~\ref{lem: 9.9}};
	\draw (4,0) node[] (1) {Prop~\ref{prop:case1}};

	\draw [-latex] (2.north east) arc(-70:250:0.5) node[midway, above] {$x=0$};
	\draw [-latex] (1.north east) arc(-70:250:0.5) node[midway, above] {all $x$};

	\draw [-latex] (2.east) -- (1.west) node[midway, above] {$x\ne 0$};

	\draw (-0.7,0) node[] {$\in$};
	\draw (-1.7,0) node[] {$(0,0,\dotsc, 0,y)\ 
$};	
	\end{tikzpicture}
	\caption{\small Diagram describing the inductive step of determining $m(s,x)$ from $m(s)$. 
	The notation $s\in A$ means that $s$ satisfies the conditions of $A$, while an arrow $A\xrightarrow{x=i} B$ means that if the sequence $s$ satisfies the conditions of $A$, then the sequence $(s,i)$ satisfies the conditions of $B$. Here $y$ is any integer in $[\overline{p}]$.}
	\label{fig:flowchart p>=5}
\end{figure}
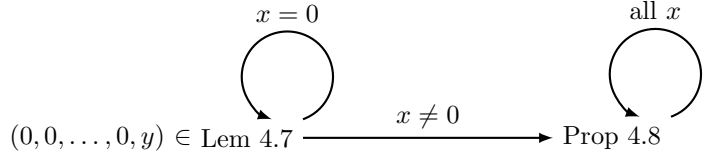
\end{rem}

We include the following example to explain how to use Lemma~\ref{lem: 9.9} and Proposition~\ref{prop:case1} (and the diagram represented in Figure 2) to compute $m(s)$ in a concrete situation. 

\begin{eg}\label{example: first}
Let $p=7$ and $s=(0100110)$. Then $(01)\in U_2(1)$ satisfies the conditions of Lemma~\ref{lem: 9.9} with $k=f=2$, so applying the lemma gives 
$$\Omega(010) = \mathcal{B}_{7^3}(7^3-7^1-1)\sqcup\{(7^3-7^1,\mu) : \mu\in\mathcal{P}(7)\setminus\{(6,1)\}^\circ \}^\circ,$$
since $\Omega(0)=\mathcal{P}(7)\setminus\{(6,1)\}^\circ$ by \cite[Theorem A]{GL1}.
Then $(010)\in U_3(1)$ satisfies the conditions of Lemma~\ref{lem: 9.9} with $k=3$ and $f=2$, giving
$$\Omega(0100) = \mathcal{B}_{7^4}(7^4-7^2-1)\sqcup\{(7^4-7^2,\mu) : \mu\in\mathcal{P}(49)\setminus\{(48,1)\}^\circ \}^\circ,$$
and applying Lemma~\ref{lem: 9.9} again to $(0100)\in U_4(1)$ with $k=4$ and $f=2$ gives
$$\Omega(01001) = \mathcal{B}_{7^5}(7^5-7^3-1)\sqcup \{(7^5-7^3,\mu) : \mu\in\mathcal{B}_{7^3}(7^3-1) \}^\circ,$$
where $\Omega(001)=\mathcal{B}_{7^3}(7^3-1)$ by Lemma~\ref{lem: 001}.
Furthermore, $$\langle \tworow{7^5}{7^5-7^3-1}\down_{P_{7^5}}, \phi(01001)\rangle \ge 2\quad\text{and}\quad \langle \hook{7^5}{7^5-7^3-1}\down_{P_{7^5}}, \phi(01001)\rangle \ge 2.$$
Thus $(01001)$ satisfies conditions (i) and (ii) of Proposition~\ref{prop:case1}. Applying the proposition, we obtain that $m(010011)=7\cdot (7^5-7^3-1)=7^6-7^4-7$ and that $(010011)$ satisfies both conditions of Proposition~\ref{prop:case1}. This exemplifies the right-hand loop in Figure~\ref{fig:flowchart p>=5}: if the sequence $t$ satisfies the conditions of Proposition~\ref{prop:case1}, then so does $(t,x)$ for all $x\in[\overline{p}]$. Applying the proposition again, we obtain $m(s)=7\cdot m(010011)=7^7-7^5-7^2$, as predicted by Theorem~\ref{thm:prime power m}. Moreover, using Theorem~\ref{thm: prime power M} we obtain that $\mathcal{B}_{7^7}(7^7-7^5-7^2)\subseteq\Omega(s)\subseteq\mathcal{B}_{7^7}(7^7-7^5)$.
\hfill$\lozenge$
\end{eg}

We can collect the results obtained so far in the following statement, which together with Theorem~\ref{thm: prime power M} gives Theorem B in the case where $n$ is a power of $p$.

\begin{thm}\label{thm: m and Omega}
Let $p\ge 5$ be a prime.
Let $s\in [\overline{p}]^k$ and $\phi(s)\in\mathrm{Lin}(P_{p^k})$. Then
$$m(s) = \begin{cases}
p^k-2 & \mathrm{if}\ \tau(s) = 1,\\ 
p^k-p^{k-f(s)}-1 & \mathrm{if}\ \tau(s) = 2,\\ 
p^k-1 & \mathrm{if}\ \tau(s) = 3,\\ 
p^k-p^{k-f(s)}-p^{k-g(s)} & \mathrm{if}\ \tau(s) = 4. 
\end{cases}$$

If $\tau(s)\in\{1,2,3\}$ then we have a complete description of $\Omega(s)$, namely
$$\Omega(s) = \begin{cases}
\mathcal{B}_{p^k}(m(s)) \sqcup \{(p^k)\}^\circ & \mathrm{if}\ \tau(s) = 1,\\
\mathcal{B}_{p^k}(m(s)) \sqcup \{(m(s)+1,\mu) : \mu\in\Omega(\triv_{P_{p^{k-f(s)}}}) \}^\circ & \mathrm{if}\ \tau(s) = 2,\\
\mathcal{B}_{p^k}(m(s)) & \mathrm{if}\ \tau(s) = 3.\\
\end{cases}$$
If $\tau(s)=4$, then $\Omega(s)\setminus\mathcal{B}_{p^k}(m(s))$ contains no thin partitions.
\end{thm}

\begin{proof}
If $\tau(s)=1$ then $\phi(s)=\triv_{P_{p^k}}$ and the statement follows from \cite[Theorem A]{GL1}. If $\tau(s)\neq 1$ then the statement follows from the combination of Lemmas~\ref{lem: 001} and~\ref{lem: 9.9} and Theorem~\ref{thm:prime power m}.
\end{proof}

Theorem~\ref{thm: m and Omega} allows us to deduce Theorem A when $n$ is a prime power.

\begin{cor}\label{cor: thm A for prime power}
Let $p\ge 5$ be a prime.
Let $s\in [\overline{p}]^k$ and $\phi(s)\in\mathrm{Lin}(P_{p^k})\setminus\{\triv_{P_{p^k}}\}$ be quasi-trivial. Then $\Omega(\phi)=\mathcal{B}_{p^k}(p^k-1),$
unless $T(\phi)=(0,1,0,0)$, in which case $m(\phi)=p^k-p^{k-f(s)}-1$ and
$$\Omega(\phi) = \mathcal{B}_n(m(\phi)) \sqcup\{(m(\phi)+1,\mu)\ |\ \mu\in\Omega(\triv_{P_{n-(m(\phi)+1)}}) \}^\circ.$$
\end{cor}
\begin{proof}
We observe that $\tau(s)\in\{2,3\}$. 
The statement now follows from Theorem~\ref{thm: m and Omega}.
\end{proof}

As mentioned in the introduction, when $p=3$ the situation is much more complicated. The following example shows that Theorem A would not hold for $p=3$. 

\begin{eg}\label{counterexample}
Let $p=3$, $k\geq 4$ and let $\phi\in\mathrm{Lin}(P_{3\6k})$ be the quasi-trivial character corresponding to the sequence $s=(1,0,0,\ldots, 0)$. Calling $a=3\6{k-1}$ one can show that 
$$\Omega(\phi)=\mathcal{B}_{3^k}(2a)\setminus \{ (2a,a-1,1), (2a,2,1^{a-2}), \tworow{3^k}{2a-1}, \hook{3^k}{2a-1} \}^\circ.$$ 
In particular we have $\tau(s)=2$, $m(s)=2a-2$ and the thin partition $(2a,a)\in \Omega(s)\smallsetminus\mathcal{B}_{3\6{k}}(m(s))$.
This is just one of the many quasi-trivial linear characters of a Sylow $3$-subgroup whose corresponding set $\Omega(\phi)$ can not be described by the statement of Theorem A.
\hfill$\lozenge$
\end{eg}

\medskip

\subsection{Proofs of Lemma \ref{lem: 9.9} and Proposition \ref{prop:case1}}\label{sec:proofs}

\begin{proof}[\bf{Proof of Lemma~\ref{lem: 9.9}}]
	\noindent\textbf{(a)} Let $f=f(s)$, $t=(s_1,\dotsc,s_f)$ and $u=(s_{f+1},\dotsc,s_k)$. We proceed by induction on $k-f$. 
		For the base case $f=k$, we have that $\Omega(s)=\mathcal{B}_{p^k}(p^k-1)$ by Lemma~\ref{lem: 001}. 
	Since $p^k>4$, we have that $\mathcal{B}_{p^{k+1}}(p^{k+1}-p-1)\subseteq\mathcal{D}\big(p,p^k,\mathcal{B}_{p^k}(p^k-1)\big)$ by Proposition~\ref{prop: BinD}. Hence $$\mathcal{B}_{p^{k+1}}(p^{k+1}-p-1)\subseteq\Omega(s,x)$$
	by Lemma~\ref{lem: DinOm}.
	On the other hand, by Theorem~\ref{thm: prime power M} we know that $M(s,x)=p^{k+1}-p$. Hence $p^{k+1}-p-1\le m(s,x) \le p^{k+1}-p$, and the statement now follows directly from Lemma~\ref{lem: 9.8}.
	
	\smallskip
	
	Now suppose that $f<k$. 
	Then $f(s^-)=f\in[k-1]$, and by the inductive hypothesis applied to $s=(s^-,s_k)$ we have that
	\begin{equation}\label{eqn:omega eqn}
	\Omega(s)=\mathcal{B}_{p^k}(p^k-p^{k-f}-1)\sqcup\{(p^{k}-p^{k-f}, \mu)\ |\ \mu\in\Omega(u)\}^\circ.
	\end{equation}
	By Proposition~\ref{prop: BinD} and Lemma~\ref{lem: DinOm}, we deduce that
	$$\mathcal{B}_{p^{k+1}}(p^{k+1}-p^{k+1-f}-p-1)\subseteq\mathcal{D}\big(p,p^k,\mathcal{B}_{p^k}(p^k-p^{k-f}-1)\big)\subseteq\mathcal{D}\big(p,p^k,\Omega(s)\big)\subseteq\Omega(s,x).$$
	
	\smallskip
	
	We now want to show that for all $r\in\{0,1,\dotsc,p-1\}$ and all $\mu\vdash p^{k+1-f}+p-r$, the partition $\lambda:=(p^{k+1}-p^{k+1-f}-p+r,\mu)$ belongs to $\Omega(s,x)$. This would allow us to conclude that $\mathcal{B}_{p^{k+1}}(p^{k+1}-p^{k+1-f}-1)\subseteq\Omega(s,x)$, since $\Omega(s,x)$ is closed under conjugation.
		
	\smallskip
	
	If $r=0$ then $\mu\vdash p^{k+1-f}+p$. Since $\Omega(u)=\mathcal{P}(p^{k-f})\setminus\{(p^{k-f}-1,1)\}^\circ$ by \cite[Theorem A]{GL1}, 
	there exists a partition $\nu_1\in\Omega(u)$ such that $\nu_1\subseteq\mu$. 
	Hence there exist partitions $\nu_2\vdash p^{k-f}+2$ and $\nu_3,\dotsc,\nu_p\vdash p^{k-f}+1$ such that $c^\mu_{\nu_1,\dotsc,\nu_p}>0$. 
	By Lemma~\ref{lem: iteratedLR} we deduce that
	$$ c^\lambda_{(p^k-p^{k-f},\nu_1), (p^k-p^{k-f}-2, \nu_2), (p^k-p^{k-f}-1,\nu_3),\dotsc,(p^k-p^{k-f}-1,\nu_p)} = c^\mu_{\nu_1,\dotsc,\nu_p} > 0.$$
	Since $p^{k-f}+2\le p^k-p^{k-f}-2$, we have that $(p^k-p^{k-f},\nu_1)$, $(p^k-p^{k-f}-2, \nu_2)$ and $(p^k-p^{k-f}-1,\nu_i)$ for all $i\in\{3,\dotsc,p\}$ are genuine partitions. Moreover, they belong to $\Omega(s)$, as shown by equation (\ref{eqn:omega eqn}), 
	so by 
	Lemma~\ref{lem: DinOm} we conclude that $\lambda\in\mathcal{D}\big(p,p^k,\Omega(s)\big)\subseteq\Omega(s,x)$, for all $x\in[\overline{p}]$.
	
	\smallskip
	
	If $r\in[p-1]$ then $\mu\vdash p^{k+1-f}+p-r$ and there exists a partition $\nu\vdash rp^{k-f}$ such that $\nu\subseteq\mu$. (If $r=1$ then we choose $\nu\in\Omega(u)$, given by \cite[Theorem A]{GL1}.) 
	By \cite[Theorem A]{GL1}, we know that
	$$\triv_{P_{rp^{k-f}}}\ \big|\ \chi^\nu\down_{P_{rp^{k-f}}}.$$
	Thus there exist $\nu_1,\dotsc,\nu_r\in\Omega(u)$ such that $c^\nu_{\nu_1,\dotsc,\nu_r}>0$, since $\phi(u)=\triv_{P_{p^{k-f}}}$. 
	Moreover, there exist partitions $\nu_{r+1},\dotsc,\nu_p\vdash p^{k-f}+1$ such that $c^\mu_{\nu_1,\dotsc,\nu_r,\nu_{r+1},\dotsc,\nu_p}>0$. 
	Using Lemma~\ref{lem: iteratedLR} we deduce that
	$$c^\lambda_{(p^k-p^{k-f},\nu_1),\dotsc,(p^k-p^{k-f},\nu_r), (p^k-p^{k-f}-1,\nu_{r+1}),\dotsc,(p^k-p^{k-f}-1,\nu_p)} = c^\mu_{\nu_1,\dotsc,\nu_p} > 0.$$
	Note that $(p^k-p^{k-f},\nu_i)$ for $i\in[r]$ and $(p^k-p^{k-f}-1,\nu_j)$ for $j\in\{r+1,\dotsc,p\}$ are indeed partitions as $p^{k-f}+1\le p^k-p^{k-f}-1$. Moreover, they belong to $\Omega(s)$ by (\ref{eqn:omega eqn}), so $\lambda\in\mathcal{D}(p,p^k,\Omega(s))\subseteq\Omega(s,x)$ for all $x\in[\overline{p}]$, by Lemma~\ref{lem: DinOm}.
	
	\smallskip
	
	Thus we have shown that $\mathcal{B}_{p^{k+1}}(p^{k+1}-p^{k+1-f}-1)\subseteq \Omega(s,x)$. 
	The statement (a) now follows from Lemma~\ref{lem: 9.8}, since $M(s,x)=p^{k+1}-p^{k+1-f}$ by Theorem~\ref{thm: prime power M}.
	
	\medskip
	
	\noindent\textbf{(b)} We turn to the proof of statement (b). Let $t=(s,x)$ and observe that $f(t)=f(s)=f$. Let $P=P_{p^{k+1}}=P_{p^k}\wr P_p$ and let $B$ be its base group, namely $P=B\rtimes P_p$ and $B\cong(P_{p^k})^{\times p}$. Let $Y=(\fS_{p^k})^{\times p}$ be the Young subgroup of $\fS_{p^{k+1}}$ naturally containing $B$. We define two further subgroups of $\fS_{p^{k+1}}$ as follows: $H:=Y\rtimes \fS_p\cong \fS_{p^k}\wr \fS_p$ and $W:=Y\rtimes P_p\cong \fS_{p^k}\wr P_p$. Clearly $P\le W\le H$. 
	
	\smallskip
		
	First, we let $\lambda=\tworow{p^{k+1}}{p^{k+1}-p^{k+1-f}-1}$ and define 
	$$\mu=\tworow{p^k}{p^k-p^{k-f}}\quad\text{and}\quad\nu=\tworow{p^k}{p^k-p^{k-f}-1}.$$
	Note that $\mu,\nu\in\Omega(s)$ by part (a) of the present lemma if $f<k$, and by Lemma~\ref{lem: 001} if $f=k$. Moreover, it is easy to see that $c^\lambda_{\mu^1,\dotsc,\mu^{p-1},\nu}=1$ where $\mu^1=\cdots=\mu^{p-1}=\mu$. 
	Since $\theta:=(\chi^\mu)^{\times (p-1)}\times\chi^\nu$ is an irreducible constituent of $\chi^\lambda\down_Y$, there exists $\rho\in\Irr(W|\theta)$ such that $\rho\mid\chi^\lambda\down_W$. But $\mu\ne\nu$, so by the description of $\Irr(\fS_{p^k}\wr P_p)$ (see Section~\ref{sec:wreath}), we have that $\rho=\theta\up^W_Y$. From \cite[Problem 5.2]{IBook}, we see that $\rho\down_P=\theta\down_B\up^P$, which has $\phi(s)^{\times p}\up^P$ as a direct summand, and hence $\langle \rho\down_P,\phi(t)\rangle\ge 1$ by Lemma~\ref{lem: easy observation} since $\phi(t)=\mathcal{X}(\phi(s);\phi_x)$. 
	On the other hand, $\mathcal{X}\big(\mu; (p-1,1)\big)\mid\chi^\lambda\down_H$ by \cite[Theorem 1.5]{DPW}. 
	Thus $\beta:=\mathcal{X}(\mu;\phi_x)$ is an irreducible constituent of $\chi^\lambda\down_W$, since $\chi^{(p-1,1)}\down^{\fS_p}_{P_p} = \sum_{i=1}^{p-1}\phi_i$, and clearly $\langle\beta\down_P,\phi(t)\rangle\ge 1$. Since $\rho\ne\beta$ are both irreducible, we find that
	$$\langle\chi^\lambda\down_P,\phi(t)\rangle \ge \langle\rho\down_P,\phi(t)\rangle + \langle\beta\down_P,\phi(t)\rangle \ge 2.$$
	
	For $\lambda=\hook{p^{k+1}}{p^{k+1}-p^{k+1-f}-1}$, a similar argument using $\mu=\hook{p^k}{p^k-p^{k-f}}$ and $\nu=\hook{p^k}{p^k-p^{k-f}-1}$, and \cite[Theorem 3.5]{GTT} to show that $\mathcal{X}(\mu;\tau)\mid\chi^\lambda\down_H$ for some $\tau\in\{(p-1,1)\}^\circ$ shows that $\langle\chi^\lambda\down_P,\phi(t)\rangle \ge 2$. Statement (b) then follows since $\chi^\lambda\down_P = \chi^{\lambda'}\down_P$.
\end{proof}

\medskip

\begin{proof}[\bf{Proof of Proposition~\ref{prop:case1}}]
Let $m=m(s)$ and let $x\in [\overline{p}]$. By Proposition~\ref{prop: BinD} and Lemma~\ref{lem: DinOm}, 
	$$\mathcal{B}_{p^{k+1}}(pm-1)\subseteq\mathcal{D}\big(p,p^k,\mathcal{B}_{p^k}(m)\big) \subseteq\mathcal{D}\big(p,p^k,\Omega(s)\big)\subseteq\Omega(s,x).$$
	
	If $\lambda=(pm,\mu)$ for some $\mu\in\mathcal{P}(p^{k+1}-pm)\setminus\{(p^{k+1}-pm)\}^\circ$, then by Proposition~\ref{prop: BinD} there exist $\nu_1,\dotsc,\nu_p\in\mathcal{P}(p^k-m)$, not all equal, such that $c^\mu_{\nu_1,\dotsc,\nu_p}\ne 0,$. 	
	Hence $c^\lambda_{(m,\nu_1),\dotsc,(m,\nu_p)} \neq 0$,
	by Lemma~\ref{lem: iteratedLR}. Since $m>\tfrac{p^k}{2}+1$, we have that $(m,\nu_i)$ is a well-defined partition for each $i$, and 
	$(m,\nu_i)\in\mathcal{B}_{p^k}(m)\subseteq\Omega(s)$. Thus $\lambda\in\mathcal{D}\big(p,p^k,\Omega(s)\big)\subseteq\Omega(s,x)$.

It remains to study the two cases where $\mu\in	\{(p^{k+1}-pm)\}^\circ$.
Suppose first that $\mu=(p^{k+1}-pm)$ and thus $\lambda=(pm,p^{k+1}-pm)=\tworow{p^{k+1}}{pm}$. Then
	$$\mathcal{X}\big(\tworow{p^k}{m};(p)\big)\ \big|\ \chi^\lambda\down_{\fS_{p^k}\wr\fS_p}^{\fS_{p^{k+1}}}$$
	by \cite[Corollary 9.1]{PW}. By hypothesis, $\langle \tworow{p^k}{m}\down_{P_{p^k}}, \phi(s)\rangle \ge 2$, and hence 
	$$\langle\mathcal{X}(\tworow{p^k}{m}\down_{P_{p^k}};\phi_0), \phi(s,x)\rangle \ge 2$$
	for all $x\in[\overline{p}]$, by Lemma~\ref{lem: 9.5}. But 
	$$\mathcal{X}(\tworow{p^k}{m}\down_{P_{p^k}};\phi_0) = \mathcal{X}\big(\tworow{p^k}{m}; (p)\big)\down_{P_{p^k}\wr P_p}^{\fS_{p^k}\wr\fS_p},$$
	hence $\langle \chi^\lambda\down_{P_{p^{k+1}}}, \phi(s,x)\rangle \ge 2$ as claimed.
Finally, if $\lambda=(pm,1^{p^{k+1}-pm})=\hook{p^{k+1}}{pm}$, then
	$$\mathcal{X}\big(\hook{p^k}{m};\nu\big)\ \big|\ \chi^\lambda\down_{\fS_{p^k}\wr\fS_p}^{\fS_{p^{k+1}}}$$
	for some $\nu\in\{(p),(1^p)\}$ by \cite[Theorem 3.5]{GTT}. By Lemma~\ref{lem: 9.5}, 
	$$\langle\mathcal{X}(\hook{p^k}{m}\down_{P_{p^k}};\phi_0), \phi(s,x)\rangle \ge 2$$
	for all $x\in[\overline{p}]$. But $\mathcal{X}(\hook{p^k}{m}\down_{P_{p^k}};\phi_0) = \mathcal{X}(\hook{p^k}{m}; \nu)\down_{P_{p^k}\wr P_p}^{\fS_{p^k}\wr\fS_p}$, so $\langle \chi^\lambda\down_{P_{p^{k+1}}}, \phi(s,x)\rangle \ge 2$ as claimed.
Since $\Omega(s,x)$ is closed under conjugation we conclude that $\mathcal{B}_{p^{k+1}}(pm)\subseteq\Omega(s,x)$. It remains to show only that $\Omega(s,x)\setminus\mathcal{B}_{p^{k+1}}(pm)$ contains no thin partitions. Let $\lambda=\tworow{p^{k+1}}{pm+a}$ for some $a\in\mathbb{N}$. 
Since $\lambda_1>pm$, Lemma~\ref{lem: LRfirstpart} implies that
for any $\mu_1,\dotsc,\mu_p\vdash p^k$ such that $c^\lambda_{\mu_1,\dotsc,\mu_p}>0$, there exists $j\in[p]$ such that $(\mu_j)_1>m$. Moreover, $\mu_j$ is thin since $\mu_j\subseteq\lambda$, so $\mu_j\notin\Omega(s)$ by hypothesis and hence $\langle \chi^\lambda\down_{(P_{p^k})^{\times p}}, \phi(s)^{\times p}\rangle = 0$. Thus $\phi(s,x)\nmid\chi^\lambda\down_{P_{p^{k+1}}}$, i.e.~$\lambda\notin\Omega(s,x)$. A similar argument holds for any hook partition $\lambda=\hook{p^{k+1}}{pm+a}$.	
\end{proof}

\smallskip

\section{Proofs of Theorems A, B and C for all natural numbers}\label{sec: new omegas arbitrary n}

Following on from the previous section, the aim of the present one is to determine the numbers $m(\phi)$ and $M(\phi)$
for all $\phi\in\Lin(P_n)$ where $n$ is now an arbitrary natural number. This allows us to complete the proofs of Theorems A, B and C. 

\smallskip

Let $n\in\mathbb{N}$ and let $n = \sum_{i=1}^t a_ip^{n_i}$ be its $p$-adic expansion, where $0\le n_1<\cdots<n_t$.
Recall that we may write $\phi=\phi(\underline{\mathbf{s}}) = \phi(\mathbf{s}(1,1))\times\cdots\times\phi(\mathbf{s}(t,a_t))$ as in Section~\ref{sec: sylow prelims} equation (\ref{eqn: index}), and recall the operator $\star$ from Definition~\ref{def:star operator}.

\begin{lemma}\label{lem: omega star omega}
	Let $p$ be any prime. For all $n\in\mathbb{N}$ and $\phi(\underline{\mathbf{s}})\in\Lin(P_n)$, $$\Omega(\underline{\mathbf{s}})=\Omega(\mathbf{s}(1,1))\star\cdots\star\Omega(\mathbf{s}(i,j))\star\cdots\star\Omega(\mathbf{s}(t,a_t)).$$
\end{lemma}

\begin{proof}
Since $P_n=(P_{p^{n_1}})^{\times a_1}\times\cdots\times (P_{p^{n_t}})^{\times a_t}\leq (\fS_{p^{n_1}})^{\times a_1}\times\cdots\times (\fS_{p^{n_t}})^{\times a_t}\leq \fS_n$ and since 
$\phi(\underline{\mathbf{s}}) = \phi(\mathbf{s}(1,1))\times\cdots\times\phi(\mathbf{s}(t,a_t))$, the statement follows from elementary properties of induction of characters. 
\end{proof}

\begin{thm}\label{thm: M all n}
	Let $p$ be an odd prime. Let $n\in\mathbb{N}$ and $\phi(\underline{\mathbf{s}})\in\Lin(P_n)$ be as above. Then
	$$M(\underline{\mathbf{s}}) = \sum_{(i,j)} M(\mathbf{s}(i,j)).$$
\end{thm}

\begin{proof}
Let $M:=\sum_{(i,j)} M(\mathbf{s}(i,j))$.
For $k\in\mathbb{N}_0$ and $s\in[\overline{p}]^k$, by Theorem~\ref{thm: prime power M} we have that $M(s)=p^k-p^{k-f(s)}$ whenever $s\neq (0,\ldots,0)$. On the other hand we have that $M(0,\dotsc,0)=p^k$ by \cite[Theorem A]{GL1}. Hence $M(\mathbf{s}(i,j))>p^{n_i}/2$ for all $(i,j)$, 
so by Lemma~\ref{lem: omega star omega} and Proposition~\ref{prop: B star B} we have that
$$\Omega(\underline{\mathbf{s}}) = \Omega(\mathbf{s}(1,1))\star\cdots\star\Omega(\mathbf{s}(t,a_t))\subseteq\mathcal{B}_{p^{n_1}}\big(M(\mathbf{s}(1,1))\big)\star\cdots\star\mathcal{B}_{p^{n_t}}\big(M(\mathbf{s}(t,a_t))\big) = \mathcal{B}_n(M).$$
Thus $M(\underline{\mathbf{s}})\le M$. On the other hand, let $\lambda^{(i,j)}\in\Omega(\mathbf{s}(i,j))$ be such that $\lambda^{(i,j)}_1 = M(\mathbf{s}(i,j))$ for each $(i,j)$ (this is possible since $\Omega(\mathbf{s}(i,j))$ is closed under conjugation). Setting $\lambda=\lambda^{(1,1)}+\cdots+\lambda^{(t,a_t)}$. It is not difficult to see that $c^\lambda_{\lambda^{(1,1)},\dotsc,\lambda^{(t,a_t)}}=1$. Hence $\lambda\in\Omega(\mathbf{s}(1,1))\star\cdots\star\Omega(\mathbf{s}(t,a_t))=\Omega(\underline{\mathbf{s}})$, and $\lambda_1=\sum_{(i,j)} \lambda^{(i,j)}_1 = M$, so $M(\underline{\mathbf{s}})\ge M$.
\end{proof}

We now aim to determine $m(\phi)$ for all $\phi\in\Lin(P_n)\setminus\{\triv_{P_n}\}$. Fix a prime $p\ge 5$. When $\phi=\triv_{P_n}$ we know by \cite[Theorem A]{GL1} that $\Omega(\triv_{P_n})=\mathcal{P}(n)$ whenever $n\in\mathbb{N}$ is not a power of $p$, while if $n=p^k$ then $\Omega(\triv_{P_{p^k}})=\mathcal{P}(p^k)\setminus\{(p^k-1,1)\}^\circ$.

As explained in Section~\ref{sec: sylow prelims}, to simplify notation we let $R=\sum_{i=1}^t a_i$ and we define the multiset $\{s_1,\dotsc,s_R\}$ as $\{s_1,\dotsc,s_R\} = \{\mathbf{s}(i,j) \mid i\in[t],\ j\in[a_i]\}$. We let $k_j$ be the length of $s_j$, so 
$\{k_1,\dotsc,k_R\}=\{n_1,\dotsc,n_t\}$ and $|\{j\in[R] \mid k_j = n_i\}|=a_i.$ When $\phi=\phi(\underline{\mathbf{s}})$ and $\underline{\mathbf{s}}$ is identified with $\{s_1,\dotsc,s_R\}$ as above, we also denote $m(\phi)$ or $m(\underline{\mathbf{s}})$ by $m(s_1,\dotsc,s_R)$. Note that the order of $s_1,\dotsc,s_R$ does not matter in determining $m(\phi)$, 
because inductions to $\mathfrak{S}_n$ of $N_{\fS_n}(P_n)$--conjugate characters of $P_n$ coincide (see Remarks~\ref{rem: N} and~\ref{rem: N2}).

\smallskip

Fix some $\phi\in\Lin(P_n)$ with corresponding multiset of sequences $\{s_1,\dotsc,s_R\}$ as described above. 
Since $P_n$ is trivial whenever $n<p$, from now on we may assume that $n\ge p$. Moreover, we may assume that $R\ge 2$ since the case of $R=1$ is treated in Section~\ref{sec: new omegas prime power}. 
In addition, we assume for the rest of this section that there exists some $i\in[R]$ such that $\tau(s_i)\ne 1$, since $\phi\ne\triv_{P_n}$. We wish to express $m(\phi)$ in terms of the quantities $m(s_1), m(s_2),\dotsc, m(s_R)$ that we determined in Section~\ref{sec: new omegas prime power}. In order to do this, we recall the following definitions (previously given in Section~\ref{sec: sylow prelims}).

\begin{defn}
	Let $n\in\mathbb{N}$ and $\phi\in\Lin(P_n)$. Suppose $\phi$ corresponds to the multiset of sequences $\{s_1,\ldots, s_R\}$. The \textit{type} of $\phi$ is the $4$-tuple $T(\phi)=(x_1,x_2,x_3,x_4)$, where for each $i\in [4]$ we set 
	$$x_i:=|\{j\in [R]\ |\ \tau(s_j)=i\}|.$$
\end{defn}

\begin{defn}
	Let $k\in\mathbb{N}_0$ and $s\in[\overline{p}]^k$. The integer $N(s)$ is defined as follows:
	$$N(s) = \begin{cases}
	p^k & \mathrm{if}\ \tau(s)=1,\\
	m(s)+1 & \mathrm{if}\ \tau(s)=2,\\
	m(s) & \mathrm{if}\ \tau(s)\in\{3,4\}.
	\end{cases}$$
	(Note that if $k=0$, then $s$ is the empty sequence and $N(s)=p^k=1$.)	
	For $\phi\in\Lin(P_n)$ as described above, let $N(\phi)$ be defined as follows:
	$$N(\phi) = \sum_{j=1}^R N(s_j).$$
\end{defn}

We are now ready to completely describe $m(\phi)$, completing the proofs of Theorems A and B (stated in Section~\ref{sec: sylow prelims}).
This is done in the following two theorems, whose proofs appear at the end of this section to improve readability.

\begin{thm}\label{thm:10.6}
	Let $p\ge 5$ be a prime and $n\in\mathbb{N}$. Let $\phi\in\Lin(P_n)$ correspond to $\{s_1,\ldots, s_R\}$ with $R\ge 2$. Suppose that $\tau(s_i)\ne 4$ for all $i\in[R]$. Then $m(\phi)=N(\phi)$ and 
	$$\Omega(\phi) = \mathcal{B}_n(m(\phi)),$$
	unless $T(\phi)=(R-1,1,0,0)$, in which case $m(\phi)=N(\phi)-1$ and
	$$\Omega(\phi) = \mathcal{B}_n(m(\phi)) \sqcup\{(m(\phi)+1,\mu)\ |\ \mu\in\Omega(\triv_{P_{p^{k_i-f(s_i)}}}) \}^\circ,$$
where $i$ is the unique element of $[R]$ such that $\tau(s_i)=2$.
\end{thm}

\begin{thm}\label{thm:10.5}
	Let $p\ge 5$ be a prime and $n\in\mathbb{N}$. Let $\phi\in\Lin(P_n)$ correspond to $\{s_1,\ldots, s_R\}$. Suppose that $\tau(s_i)=4$ for some $i\in[R]$. Then
	\begin{itemize}\setlength\itemsep{0.5em}
		\item[(i)] $m(\phi)=N(\phi)$, and
		\item[(ii)] $\Omega(\phi)\setminus\mathcal{B}_n(m(\phi))$ contains no thin partitions.
	\end{itemize}
\end{thm}

We remark that Theorem~\ref{thm:10.6} in fact holds for $\phi=\triv_{P_n}$ as well, since $\Omega(\triv_{P_n})=\mathcal{P}(n)$ for $p\ge 5$ and $R\ge 2$ by \cite{GL1}.
We also notice that Theorem~\ref{thm:10.6} implies Theorem A. This follows by observing that if $\phi$ is quasi-trivial then $\tau(s_i)\neq 4$ for all $i\in [R]$.
Moreover, when $T(\phi)=(R-1,1,0,0)$ then $p\6{k_i-f(s_i)}=n-(m(\phi)+1)$.
Furthermore, Theorems~\ref{thm: M all n}, \ref{thm:10.6} and~\ref{thm:10.5} clearly complete the proof of Theorem B in the case where $n$ is not a power of $p$.

We refer the reader to Example~\ref{example: new omega thms} below for an illustration of the results of Theorems~\ref{thm:10.6} and~\ref{thm:10.5} in concrete examples.

\smallskip

Before proving Theorems~\ref{thm:10.6} and~\ref{thm:10.5} we show how to use them to obtain a proof of Theorem C from the introduction. 

\begin{proof}[Proof of Theorem C]
Theorem~\ref{thm: m and Omega} shows that for all $k\in\mathbb{N}$ and $s\in[\overline{p}]^k$, we have $N(s)\ge m(s)>\tfrac{p^k}{2}$. Using this together with Theorems~\ref{thm:10.6} and~\ref{thm:10.5} we deduce that $m(\phi)\ge \sum_{i=1}^R N(s_i) > \tfrac{n}{2}$ for every $\phi\in\Lin(P_n)$.
It follows that $\mathcal{B}_n(\tfrac{n}{2})\subseteq\Omega_n$.
This in particular implies 
Theorem C, using the following classical result of Erd\H{o}s and Lehner \cite[(1.4)]{EL}: if $f(n)$ is any function such that $f(n)\to\infty$ as $n\to\infty$, then for all but $o(|\mathcal{P}(n)|)$ partitions $\lambda$ of $n$, the quantities $\lambda_1$ and $l(\lambda)$ lie between $\sqrt{n}\cdot( \tfrac{\log n}{c} \pm f(n) )$ where $c$ is a constant.
\end{proof}

\medskip

\subsection{Proofs of Theorems~\ref{thm:10.6} and~\ref{thm:10.5}}

Let $n\in\mathbb{N}$ and $\phi\in\Lin(P_n)$ be as described after the proof of Theorem~\ref{thm: M all n}. The proof of Theorem~\ref{thm:10.6} follows Lemma~\ref{lem: 10.7} below.
\begin{lemma}\label{lem: 10.7}
	Suppose that
	$T(\phi)=(R-1,1,0,0)$
	and let $i\in[R]$ be such that $\tau(s_i)=2$. Then
	$$\Omega(\phi) = \mathcal{B}_n(N(\phi)-1) \sqcup\{(N(\phi),\mu)\ |\ \mu\in\Omega(\triv_{P_{p^{k_i-f(s_i)}}}) \}^\circ.$$
\end{lemma}

\begin{proof}
Let $m=n-p^{k_i}$, so $\phi=\phi(s_1)\times\cdots\times\phi(s_R) = \triv_{P_m}\times\phi(s_i)$. Since $R\ge 2$ and $\tau(s_i)=2$, we have that $m\neq 0$ and $k_i\ge 2$. To ease the notation, let $k=k_i$, $s=s_i$ and $f=f(s_i)$. 

By Lemma~\ref{lem: omega star omega}, we have that $\Omega(\phi)=\Omega(\triv_{P_m})\star\Omega(s)$ and $\Omega(\triv_{P_m})$ is determined in \cite[Theorem A]{GL1}. Moreover, by Corollary~\ref{cor: thm A for prime power} we have that
$$\Omega(s) = \mathcal{B}_{p^k}(p^k-p^{k-f}-1)\sqcup\{(p^k-p^{k-f},\mu)\ |\ \mu\in\Omega(\triv_{P_{p^{k-f}}}) \}^\circ,$$
so in particular 
\begin{equation}\label{eqn: omega(s) sandwich}
\mathcal{B}_{p^k}(m(s))\subseteq\Omega(s)\subseteq\mathcal{B}_{p^k}(m(s)+1). 
\end{equation}

\noindent\emph{Case 1: if $m$ is not a power of $p$, then $\Omega(\triv_{P_m})=\mathcal{P}(m)$.} Since $m(s)>\tfrac{p^k}{2}$, applying $\mathcal{P}(m)\star-$ to \eqref{eqn: omega(s) sandwich}, we obtain
$$\mathcal{B}_n\big(N(\phi)-1\big) \subseteq\Omega(\phi)\subseteq \mathcal{B}_n\big(N(\phi)\big)$$
by Proposition~\ref{prop: B star B}, as $N(\phi)=m+p^k-p^{k-f}=n-p^{k-f}$. Since $\Omega(\phi)^\circ=\Omega(\phi)$, it suffices to 
show that given $\mu\in\mathcal{P}(p^{k-f})$, then $(N(\phi), \mu)\in\Omega(\phi)$ if and only if $\mu\in\Omega(\triv_{P_{p^{k-f}}})$.

So fix a partition $\lambda\vdash n$ such that $\lambda_1=N(\phi)$.
If $\lambda\in\Omega(\phi)=\mathcal{P}(m)\star\Omega(s)$, then $c^\lambda_{\alpha\beta}>0$ for some $\alpha\vdash m$ and $\beta\in\Omega(s)$. Hence $\lambda_1\le\alpha_1+\beta_1\le m+(p^k-p^{k-f})=N(\phi)$, so in fact this holds with equality. Thus $\alpha=(m)$ and $\beta=(\beta_1,\mu)$ where $\beta_1=p^k-p^{k-f}$, and $\mu\in\Omega(\triv_{P_{p^{k-f}}})$ since $\beta\in\Omega(s)$. Moreover, $\lambda_1=\alpha_1+\beta_1$ and $\alpha=(m)$ together imply that $\beta=(\lambda_1-m,\lambda_2,\lambda_3,\dotsc)$ and that $\lambda=(N(\phi),\mu)$.

Conversely, if $\lambda=(N(\phi),\mu)$ for some $\mu\in\Omega(\triv_{P_{p^{k-f}}})$, then clearly $\lambda\in\mathcal{P}(m)\star\Omega(s)=\Omega(\phi)$ since 
$c\6\lambda_{(m), (p^k-p^{k-f},\mu)}\neq 0$, and thus the set $\Omega(\phi)$ is as claimed.

\smallskip

\noindent\emph{Case 2: if $m=p^l$ for some $l\in\mathbb{N}$, then $\Omega(\triv_{P_m})=\mathcal{P}(p^l)\setminus\{(p^l-1,1)\}^\circ$.}
We have that
$$\Omega(\phi)=\Omega(\triv_{P_{p^l}})\star\Omega(s)\subseteq\mathcal{P}(p^l)\star\mathcal{B}_{p^k}(m(s)+1)=\mathcal{B}_n\big(N(\phi)\big),$$
by Proposition~\ref{prop: B star B}. On the other hand, by Lemma~\ref{lem: 10.3} we have that
$$\Omega(\phi) = \Omega(\triv_{P_{p^l}}) \star\Omega(s)\supseteq \big(\mathcal{P}(p^l)\setminus\{(p^l-1,1)\}^\circ\big) \star\mathcal{B}_{p^k}(p^k-p^{k-f}-1) = \mathcal{B}_n\big(N(\phi)-1\big).$$	
Arguing exactly as in Case 1 we deduce that $\Omega(\phi)=\mathcal{B}_n(N(\phi)-1) \sqcup\{(N(\phi),\mu)\ |\ \mu\in\Omega(\triv_{P_{p^{k-f}}}) \}^\circ$, as required.
\end{proof}

\begin{proof}[Proof of Theorem~\ref{thm:10.6}]
The case $T(\phi)=(R-1,1,0,0)$ is treated in Lemma~\ref{lem: 10.7}.
	We may now assume that $T(\phi)\neq (R-1,1,0,0)$. That is, $s_1,\dotsc,s_R$ are such that \emph{either} there exists $i\in[R]$ with $\tau(s_i)=3$, \emph{or} there exists $i\ne j\in[R]$ with $\tau(s_i)=\tau(s_j)=2$ and $\tau(s_l)\in\{1,2\}$ for all $l\in[R]$. 
	We proceed by induction on $R$.
	
	We begin with the base case $R=2$. 
	Since we may reorder the $s_i$ without loss of generality, we may assume that
	$$\big( \tau(s_1), \tau(s_2) \big) \in\{ (1,3), (2,3), (3,3), (2,2)\}.$$
	The arguments in each case are similar, but for clarity we will treat each one separately. To ease the notation we let $k=k_1$, $f=f(s_1)$ (if $\tau(s_1)\ne 1$), $l=k_2$ and $e=f(s_2)$. 
	
	\smallskip
	
	\noindent\textbf{If $\big(\tau(s_1),\tau(s_2)\big)=(1,3)$:} we have that 
	$$\Omega(\phi) = \begin{cases}\mathcal{P}(1)\star\mathcal{B}_{p^l}(p^l-1) & \mathrm{if}\ k=0,\\
	(\mathcal{P}(p^k)\setminus\{(p^k-1,1)\}^\circ)\star\mathcal{B}_{p^l}(p^l-1) & \mathrm{otherwise},\end{cases}$$
	which equals $\mathcal{B}_n\big(N(\phi)\big)$ in each instance by Proposition~\ref{prop: B star B} and Lemma~\ref{lem: 10.3} respectively, as $N(\phi)=p^k+p^l-1$.
	
	\smallskip
	
	\noindent\textbf{If $\big(\tau(s_1),\tau(s_2)\big)=(2,3)$:} by Corollary~\ref{cor: thm A for prime power} we have $\mathcal{B}_{p^k}(m(s_1))\subseteq\Omega(s_1)\subseteq\mathcal{B}_{p^k}(m(s_1)+1)$. Thus Proposition~\ref{prop: B star B} shows that  
	$$\mathcal{B}_n\big(N(\phi)-1\big)\subseteq\Omega(\phi)\subseteq\mathcal{B}_n\big(N(\phi)\big)$$
	since $N(\phi)=m(s_1)+1+m(s_2)=p^k-p^{k-f}+p^l-1$. Let $\lambda=(N(\phi),\mu)$ where $\mu$ is any partition of $p^{k-f}+1$. Since $p^{k-f}+1$ is not a power of $p$, then
	$\Omega(\triv_{P_{p^{k-f}+1}})=\mathcal{P}(p^{k-f}+1)$ by \cite[Theorem A]{GL1}. But $\triv_{P_{p^{k-f}+1}} = \triv_{P_{p^{k-f}}}\times\triv_{P_1}$, so $\mu\in\Omega(\triv_{P_{p^{k-f}+1}}) = \Omega(\triv_{P_{p^{k-f}}})\star\Omega(\triv_{P_1})$. That is, there exists $\nu\in\Omega(\triv_{P_{p^{k-f}}})$ such that $c^\mu_{\nu,(1)}>0$. Then by Lemma~\ref{lem: iteratedLR},
	$$c^{(N(\phi),\mu)}_{(p^k-p^{k-f},\nu), (p^l-1,1)} = c^\mu_{\nu,(1)}>0,$$
	and thus $\lambda\in\Omega(s_1)\star\Omega(s_2)=\Omega(\phi)$. Since $\Omega(\phi)^\circ=\Omega(\phi)$, we have that $\Omega(\phi)=\mathcal{B}_n\big(N(\phi)\big)$.
	
	\smallskip
	
	\noindent\textbf{If $\big(\tau(s_1),\tau(s_2)\big)=(3,3)$:} then $\Omega(\phi)=\mathcal{B}_{p^k}(p^k-1)\star\mathcal{B}_{p^l}(p^l-1) =\mathcal{B}_n\big(N(\phi)\big)$, by Proposition~\ref{prop: B star B}.
	
	\smallskip
	
	\noindent\textbf{If $\big(\tau(s_1),\tau(s_2)\big)=(2,2)$:} then clearly $\mathcal{B}_n\big(N(\phi)-2\big) = \mathcal{B}_{p^k}(m(s_1))\star\mathcal{B}_{p^l}(m(s_2)) \subseteq\Omega(\phi)$ and $\Omega(\phi)\subseteq \mathcal{B}_{p^k}(m(s_1)+1)\star\mathcal{B}_{p^l}(m(s_2)+1) = \mathcal{B}_n\big(N(\phi)\big)$, by Proposition~\ref{prop: B star B}. Since $\Omega(\phi)^\circ=\Omega(\phi)$, in order to show $\mathcal{B}_n\big(N(\phi)\big) = \Omega(\phi)$ it remains to prove that
	$$\lambda=(N(\phi)-2+j,\mu)\ \in\ \Omega(\phi),\ \text{for all}\  j\in\{1,2\}\ \text{and for all}\ \mu\vdash p^{k-f}+p^{l-e}+2-j.$$ 
	Fix some $\mu\vdash p^{k-f}+p^{l-e}+2-j$ and consider $\lambda=(N(\phi)-2+j,\mu)$.
	Clearly $|\mu|$ is not a power of $p$, so 
	$$\mu\in\mathcal{P}(|\mu|) = \Omega(\triv_{P_{|\mu|}}) = \Omega(\triv_{P_{p^{k-f}}})\star\Omega(\triv_{P_{p^{l-e}+2-j}}).$$
	That is, there exist $\nu\in\Omega(\triv_{P_{p^{k-f}}})$ and $\omega\in\Omega(\triv_{P_{p^{l-e}+2-j}})$ such that $c^\mu_{\nu,\omega}>0$. Then by Lemma~\ref{lem: iteratedLR}
	$$ c^{(N(\phi)-2+j,\mu)}_{(p^k-p^{k-f},\nu), (p^l-p^{l-e}-2+j,\omega)} = c^\mu_{\nu,\omega}>0.$$
Thus $\lambda\in\Omega(\phi)$ and hence $\mathcal{B}_n\big(N(\phi)\big)=\Omega(\phi)$ in all cases when $R=2$.
	
	\smallskip
	
	Now for the inductive step: let $R\ge 3$ and suppose that the statement of the theorem holds for $R-1$. Since $T(\phi)\neq (R-1,1,0,0)$, there exists $i\in [R]$ such that the multiset $\{s_1,\dotsc,s_{i-1},s_{i+1},\dotsc,s_R\}$ corresponds to a linear character $\psi\in\mathrm{Lin}(P_{n-p^{k_i}})$ with $T(\psi)\neq (R-2,1,0,0)$. Without loss of generality, let $i=1$. Let $k=k_1$, $s=s_1$, $f=f(s_1)$ (if $\tau(s_1)\ne 1$) and let $\psi\in\Lin(P_{n-p^k})$ be the above linear character such that $\phi=\phi(s)\times\psi$. Then $\Omega(\phi)=\Omega(s)\star\Omega(\psi)$ and $\Omega(\psi)=\mathcal{B}_{n-p^k}\big(N(\psi)\big)$ by the inductive hypothesis. In order to show that $\Omega(\phi)=\mathcal{B}_n\big(N(\phi)\big)$, we split into cases depending on $\tau(s)\in\{1,2,3\}$.
	
	\smallskip
	
	\noindent\textbf{If $\tau(s)=1$:} then
	$$\Omega(\phi) = \begin{cases}\mathcal{P}(1)\star\mathcal{B}_{n-p^k}\big(N(\psi)\big) & \mathrm{if}\ k=0,\\
	(\mathcal{P}(p^k)\setminus\{(p^k-1,1)\}^\circ)\star\mathcal{B}_{n-p^k}\big(N(\psi)\big) & \mathrm{otherwise}.\end{cases}$$
Hence $\Omega(\phi)=\mathcal{B}_n\big(N(\phi)\big)$ by Proposition~\ref{prop: B star B} and Lemma~\ref{lem: 10.3}, as $N(\phi)=p^k+N(\psi)$.
	
	\smallskip
	
	\noindent\textbf{If $\tau(s)=2$:} then
	$$\mathcal{B}_n\big(N(\phi)-1\big)\subseteq\Omega(\phi)\subseteq\mathcal{B}_n\big(N(\phi)\big),$$
	where $N(\phi)=m(s)+1+N(\psi)=p^k-p^{k-f}+N(\psi)$. Since $\Omega(\phi)^\circ=\Omega(\phi)$, it suffices to show that
	$$\lambda=(N(\phi),\mu)\ \in\ \Omega(\phi)\ \text{for all}\ \mu\in\mathcal{P}(n-N(\phi))=\mathcal{P}(n-p^k+p^{k-f}-N(\psi)).$$
	Fix such a partition $\lambda$. Since $p^{k-f}\ge p\ge 5$ then by Lemma~\ref{lem: 10.3} we have that
	$$\mu\in\mathcal{P}\big(n-p^k+p^{k-f}-N(\psi)\big) = \big( \mathcal{P}(p^{k-f})\setminus\{(p^{k-f}-1,1)\}^\circ \big) \star \mathcal{P}(n-p^k-N(\psi)).$$
 Thus there exist $\nu\in \mathcal{P}(p^{k-f})\setminus\{(p^{k-f}-1,1)\}^\circ$ and $\omega\in \mathcal{P}(n-p^k-N(\psi))$ such that $c^\mu_{\nu,\omega}>0$. This shows by Lemma~\ref{lem: iteratedLR} that
	$$c^{(N(\phi),\mu)}_{(p^k-p^{k-f},\nu), (N(\psi),\omega)} = c^\mu_{\nu,\omega}>0,$$
	and so $\lambda\in\Omega(s)\star\mathcal{B}_n\big(N(\psi)\big)=\Omega(\phi)$ as required. (Here we used that $N(\psi)\lneqq n-p^k$. This holds because $T(\psi)\neq (R-1,0,0,0)$ and hence $\psi\neq \triv_{P_{n-p^k}}$.)
	
	\smallskip
	
	\noindent\textbf{If $\tau(s)=3$:} then $\Omega(\phi)=\mathcal{B}_{p^k}(p^k-1)\star\mathcal{B}_{n-p^k}\big(N(\psi)\big) = \mathcal{B}_n\big(N(\phi)\big)$ by Proposition~\ref{prop: B star B}.
	
	Hence $\Omega(\phi)=\mathcal{B}_n\big(N(\phi)\big)$ in all cases.
\end{proof}

\begin{lemma}\label{lem: 10.8}
	For $i\in\{1,2\}$, let $n_i,m_i\in\mathbb{N}$ be such that $\tfrac{n_i}{2}<m_i\le n_i$. Furthermore, let $\Delta_i\subseteq\mathcal{P}(n_i)$ be such that $\mathcal{B}_{n_i}(m_i)\subseteq\Delta_i$ and $\Delta_i\setminus\mathcal{B}_{n_i}(m_i)$ contains no thin partitions. Then
	$$\mathcal{B}_{n_1+n_2}(m_1+m_2) \subseteq \Delta_1\star\Delta_2$$
	and $(\Delta_1\star\Delta_2)\setminus\mathcal{B}_{n_1+n_2}(m_1+m_2)$ contains no thin partitions.
\end{lemma}

\begin{proof}
	By Proposition~\ref{prop: B star B}, we know that $\mathcal{B}_{n_1+n_2}(m_1+m_2) \subseteq \Delta_1\star\Delta_2$.
	
	First, suppose $\lambda\in (\Delta_1\star\Delta_2)\setminus\mathcal{B}_{n_1+n_2}(m_1+m_2)$ satisfies $l(\lambda)\le 2$. Then $\lambda_1>m_1+m_2$. But $\lambda\in\Delta_1\star\Delta_2$ implies that $c^\lambda_{\mu,\nu}>0$ for some $\mu\in\Delta_1$ and $\nu\in\Delta_2$. Thus $\mu_1+\nu_1\ge\lambda_1$ by Lemma~\ref{lem: LRfirstpart}, giving either $\mu_1>m_1$ or $\nu_1>m_2$. However, $\mu,\nu\subseteq\lambda$ so $l(\mu),l(\nu)\le l(\lambda)\le 2$. That is, both $\mu$ and $\nu$ are thin but either $\mu\in\Delta_1\setminus\mathcal{B}_{n_1}(m_1)$ or $\nu\in\Delta_2\setminus\mathcal{B}_{n_2}(m_2)$, a contradiction.
	A similar argument shows that $(\Delta_1\star\Delta_2)\setminus\mathcal{B}_{n_1+n_2}(m_1+m_2)$ contains no other thin partitions.
%
\end{proof}

We are now ready to prove Theorem~\ref{thm:10.5}.

\begin{proof}[Proof of Theorem~\ref{thm:10.5}]
	We show that $\mathcal{B}_n\big(N(\phi)\big)\subseteq\Omega(\phi)$ and that $\Omega(\phi)\setminus\mathcal{B}_n\big(N(\phi)\big)$ contains no thin partitions. From this we also deduce that $m(\phi)=N(\phi)$. 
	
	We proceed by induction on $R$, beginning with the base case $R=2$. (The case of $R=1$ follows from Theorem~\ref{thm: m and Omega}.)
	Without loss of generality we may assume that $\tau(s_2)=4$. Let $k=k_1$, $f=f(s_1)$ (if $\tau(s_1)\ne 1$) and let $l=k_2$. By Lemma~\ref{lem: omega star omega}, we know that $\Omega(\phi)=\Omega(s_1)\star\Omega(s_2)$, and from Theorem~\ref{thm: m and Omega} we know that $\Omega(s_2)\setminus\mathcal{B}_{p^l}(m(s_2))$ contains no thin partitions. We split into cases according to $\tau(s_1)\in\{1,2,3,4\}$.
	
	\smallskip
	
	\noindent\textbf{(i) If $\tau(s_1)=1$:} then $N(\phi)=p^k+m(s_2)$. If $k=0$, then Proposition~\ref{prop: B star B} implies that
	$$\Omega(\phi) = \mathcal{P}(1)\star\Omega(s_2) \supseteq \mathcal{P}(1)\star \mathcal{B}_{p^l}(m(s_2)) = \mathcal{B}_n\big(N(\phi)\big).$$
	Moreover, $\Omega(\phi)\setminus\mathcal{B}_n\big(N(\phi)\big)$ contains no thin partitions by Lemma~\ref{lem: 10.8}. Otherwise, if $k\ge 1$ then
	$$\Omega(\phi)\supseteq \big(\mathcal{P}(p^k)\setminus\{(p^k-1,1)\}^\circ \big) \star\mathcal{B}_{p^l}(m(s_2)) = \mathcal{B}_n\big(N(\phi)\big)$$
	by Lemma~\ref{lem: 10.3}. 
	Suppose $\lambda\in\Omega(\phi)\setminus\mathcal{B}_n\big(N(\phi)\big)$ satisfies $l(\lambda)\le 2$, so $\lambda_1>N(\phi)=p^k+m(s_2)$. Then $c^\lambda_{\mu,\nu}>0$ for some $\mu\in\Omega(s_1)$ and $\nu\in\Omega(s_2)$, and $\lambda_1\le \mu_1+\nu_1$ by Lemma~\ref{lem: LRfirstpart}. But $\mu_1\le p^k$, so $\nu_1>m(s_2)$. However, $\nu\subseteq\lambda$ so $l(\nu)\le 2$, contradicting $\nu\in\Omega(s_2)\setminus\mathcal{B}_{p^l}(m(s_2))$. Also there cannot be any $\lambda\in\Omega(\phi)\setminus\mathcal{B}_n\big(N(\phi)\big)$ such that $\lambda_1\le 2$, since $\Omega(\phi)$ and $\mathcal{B}_n\big(N(\phi)\big)$ are both closed under conjugation. A similar argument shows that there are no hooks in $\Omega(\phi)\setminus\mathcal{B}_n\big(N(\phi)\big)$. 
	
	\smallskip
	
	\noindent\textbf{(ii) If $\tau(s_1)=2$:} then by Theorem~\ref{thm: m and Omega} we have that
	$$\Omega(\phi)\supseteq\mathcal{B}_{p^k}(p^k-p^{k-f}-1)\star\mathcal{B}_{p^l}(m(s_2)) = \mathcal{B}_n\big(N(\phi)-1\big).$$
	Let $\lambda=(N(\phi),\mu)$ where $\mu\in \mathcal{P}(n-N(\phi))$. Using  Lemma~\ref{lem: 10.3} we observe that $\mathcal{P}(n-N(\phi)) = \Omega(\triv_{P_{p^{k-f}}})\star\mathcal{P}(p^l-m(s_2))$. Hence
	$ c^\lambda_{(p^k-p^{k-f},\nu),(m(s_2),\omega)} = c^\mu_{\nu,\omega}>0$
	for some $\nu\in \Omega(\triv_{P_{p^{k-f}}})$ and $\omega\vdash p^l-m(s_2)$, by Lemma~\ref{lem: iteratedLR}.
	Thus $\lambda\in\Omega(s_1)\star\Omega(s_2)=\Omega(\phi)$, and hence $\mathcal{B}_n\big(N(\phi)\big)\subseteq\Omega(\phi)$ since $\Omega(\phi)^\circ=\Omega(\phi)$.
	If $\lambda\in\Omega(\phi)\setminus\mathcal{B}_n\big(N(\phi)\big)$ satisfies $l(\lambda)\le 2$, then $c^\lambda_{\mu,\nu}>0$ for some $\mu\in\Omega(s_1)$ and $\nu\in\Omega(s_2)$. Since $\mu_1\le p^k-p^{k-f}$ we must have $v_1>m(s_2)$, as $\lambda_1>N(\phi)=p^k-p^{k-f}+m(s_2)$. But then $\nu\in\Omega(s_2)\setminus\mathcal{B}_{p^l}(m(s_2))$ and $l(\nu)\le 2$, a contradiction. A similar argument shows that $\Omega(\phi)\setminus\mathcal{B}_n\big(N(\phi)\big)$ contains no other thin partitions.
	
		\smallskip
	
	\noindent\textbf{(iii) If $\tau(s_1)\in\{3,4\}$:} then the assertions follow from Proposition~\ref{prop: B star B} and Lemma~\ref{lem: 10.8}.
	
	\smallskip	
	
	Finally, we turn to the inductive step. Assume $R\ge 3$ and that the statement of the theorem holds for $R-1$. Let $k=k_1$, and let $\psi\in\Lin(P_{n-p^k})$ be such that $\phi=\phi(s_1)\times\psi$, so $\psi$ corresponds to $s_2,\dotsc,s_R$ and $\Omega(\phi)=\Omega(s_1)\star\Omega(\psi)$. We distinguish two cases, depending on the type $T(\phi)$.

Suppose that $T(\phi)=(R-2,1,0,1)$. Since $R\ge 3$, we may without loss of generality assume that $\tau(s_1)=1$. By the inductive hypothesis, $m(\psi)=N(\psi)$ and $\Omega(\psi)\setminus\mathcal{B}_{n-p^k}\big(N(\psi)\big)$ contains no thin partitions. Then $\mathcal{B}_n\big(N(\phi)\big)\subseteq\Omega(\phi)$ and arguing as in case (i) above we obtain that $\Omega(\phi)\setminus\mathcal{B}_n\big(N(\phi)\big)$ contains no thin partitions.
	
Suppose now that $T(\phi)\neq (R-2,1,0,1)$. In this case we may without loss of generality assume that $\tau(s_1)=4$.
	If $|\{i\in\{2,3,\dotsc,R\}\ |\ \tau(s_i)=4\}|=0$ then the first part of Theorem~\ref{thm:10.6} 
	gives us that $\Omega(\psi)=\mathcal{B}_{n-p^k}\big(N(\psi)\big)$. The required results then follow from Proposition~\ref{prop: B star B} and Lemma~\ref{lem: 10.8}.
	On the other hand, if $|\{i\in\{2,3,\dotsc,R\}\ |\ \tau(s_i)=4\}|>0$, then by the inductive hypothesis we have that $m(\psi)=N(\psi)$ and $\Omega(\psi)\setminus\mathcal{B}_{n-p^k}\big(N(\psi)\big)$ contains no thin partitions. The required results then also follow from Proposition~\ref{prop: B star B} and Lemma~\ref{lem: 10.8}.
\end{proof}

We conclude this section with examples illustrating the main theorems.

\begin{eg}\label{example: new omega thms}
	Let $p=5$. We consider (i) $n=25$, (ii) $n=125$ and (iii) $n=175$.
	
	\smallskip
	
	\noindent\textbf{(i) $n=25$.} 
	We describe $\Omega(\phi)$ completely for all $\phi=\phi(s)\in\Lin(P_{25})$ using Theorem~\ref{thm: m and Omega} and \cite[Theorem A]{GL1}. This is summarised in Table~\ref{table:n=25} below, where each $*$ represents any element of $\{1,2,\dots,p-1\}$ and $\mathcal{P}'(m):=\mathcal{P}(m)\setminus\{(m-1,1)\}^\circ = \mathcal{B}_m(m-2)\sqcup\{(m)\}^\circ$ for $m\in\mathbb{N}_{\ge 5}$.
	
	\begin{table}[h]
		\centering
		\[ \def\arraystretch{1.2}
		\begin{array}{|c||ccc|ccc|}
		\hline
		s & \text{type } \tau(s) & f(s) && m(s) & M(s) & \Omega(s)\\
		\hline
		(0,0) & 1 & \text{n/a} && 23 & 25 & \mathcal{P}'(25)\\
		(0,*) & 3 & 2 && 24 & 24 & \mathcal{B}_{25}(24)\\
		(*,0) & 2 & 1 && 19 & 20 & \mathcal{B}_{25}(19)\sqcup\{(20,\mu)\mid \mu\in\mathcal{P}'(5)\}^\circ\\ 
		(*,*) & 4 & 1 && 19 & 20 & \mathcal{B}_{25}(19)\sqcup\{(20,\mu)\mid \mu\in\mathcal{B}_5(4) \}^\circ\\
		\hline
		\end{array}\]
		\caption{Data on $\Omega(\phi)$ for $\phi=\phi(s)\in\Lin(P_{25})$.}\label{table:n=25}
	\end{table}
	
	We can similarly determine $\Omega(\phi)$ explicitly for all $\phi\in\Lin(P_{p^2})$, for all $p\ge 5$.
	\medskip
	
	\noindent\textbf{(ii) $n=125$.} 
	The various $\Omega(s)$ are summarised in Table~\ref{table:n=125} below.
	
	\begin{table}[h]
		\centering
		\begin{small}
			\[ \def\arraystretch{1.2}
			\begin{array}{|c||ccccc|ccc|}
			\hline
			s & \tau(s) & f(s) & g(s) & \eta(s) && m(s) & M(s) & \Omega(s)\\
			\hline
			(0,0,0) & 1 & \text{n/a} & \text{n/a} & \text{n/a} && 123 & 125 & 
			\mathcal{P}'(125)\\
			(0,0,*) & 3 & 3 & \text{n/a} & \text{n/a} && 124 & 124 & \mathcal{B}_{125}(124)\\
			(0,*,0) & 2 & 2 & \text{n/a} & \text{n/a} && 119 & 120 &  \mathcal{B}_{125}(119)\sqcup \{(120,\mu)\mid \mu\in\mathcal{P}'(5) \}^\circ\\
			(0,*,*) & 4 & 2 & 3 & 119 && 119 & 120 & \mathcal{B}_{125}(119)\sqcup\{(120,\mu)\mid \mu\in\mathcal{B}_5(4) \}^\circ\\
			(*,0,0) & 2 & 1 & \text{n/a} & \text{n/a} && 99 & 100 & \mathcal{B}_{125}(99)\sqcup \{(100,\mu)\mid \mu\in\mathcal{P}'(25) \}^\circ\\
			(*,0,*) & 4 & 1 & 3 & 99 && 99 & 100 & \mathcal{B}_{125}(99)\sqcup \{(100,\mu)\mid\mu\in\mathcal{B}_{25}(24) \}^\circ\\
			(*,*,0) & 4 & 1 & 2 & 95 && 95 & 100 & \text{(see below)}\\
			(*,*,*) & 4 & 1 & 2 & 95 && 95 & 100 & \text{(see below)}\\
			\hline
			\end{array}\]
		\end{small}
		\caption{Data on $\Omega(\phi)$ for $\phi=\phi(s)\in\Lin(P_{125})$.}\label{table:n=125}
	\end{table}
	
	Recall from Theorem~\ref{thm: m and Omega} that we know $\Omega(s)$ exactly whenever $\tau(s)\ne 4$. We are able to determine $\Omega(s)$ completely for $s=(0,*,*)$ and $s=(*,0,*)$ even though $\tau(s)=4$ because $M(s)=m(s)+1$ in these cases, so the result follows from Lemma~\ref{lem: 9.8}.
	
	In the remaining instances when $\tau(s)=4$, i.e.~for $s=(*,*,0)$ and $s=(*,*,*)$, then $\mathcal{B}_{125}(95)\subseteq\Omega(s)\subseteq\mathcal{B}_{125}(100)$ and $\Omega(s)\setminus\mathcal{B}_{125}(95)$ contains no thin partitions, by Proposition~\ref{prop:case1}. (In other words, $\Omega(s)$ does not contain $(95+i,30-i)$, $(95+i,1^{30-i})$ or their conjugates for any $i\in[5]$.) Moreover,
	$$\Omega(*,*,x)\cap\{(100,\mu)\mid\mu\vdash 25 \}^\circ = \{(100,\mu)\mid \mu\in\Omega(*,x) \}^\circ$$ 
	for all $x\in\{0,1,\dotsc,4\}$ where $\Omega(*,x)$ has already been determined in (i) above.
	
	\medskip
	\noindent\textbf{(iii) $n=175$.} 
Since $175=5\63+2\cdot5\62$, each linear character $\phi(s)$ is labelled by a sequence 
$s=(s_1,s_2,s_3)$ where $s_1\in[\overline{5}]^3$ and $s_2,s_3\in[\overline{5}]^2$.
%
%
	To give some examples, we list $\Omega(s)$ when $s_1=(0,0,0)$ in Table~\ref{table:n=175} below.\hfill$\lozenge$
	
	\begin{table}[h]
		\centering
		\begin{small}
			\[ \def\arraystretch{1.2}
			\begin{array}{|c|c|c|c|}
			\hline
			s_2,\ s_3 & \tau(s_i)_{i=1,2,3} & N(s_i)_{i=1,2,3}; N(\phi) & \Omega(s)\\
			\hline
			(0,0), (0,0) & 1,1,1 & 125,25,25; 175 & \mathcal{P}(175)\\
			(0,0), (*,0) & 1,1,2 & 125,25,20; 170 & \mathcal{B}_{175}(169)\sqcup \{(170,\mu)\mid \mu\in\mathcal{P}'(5) \}^\circ\\
			(0,0), (0,*) & 1,1,3 & 125,25,24; 174 & \mathcal{B}_{175}(174)\\
			(*,0),(*,0) & 1,2,2 & 125,20,20; 165 & \mathcal{B}_{175}(165)\\
			(*,0),(0,*) & 1,2,3 & 125,20,24; 169 & \mathcal{B}_{175}(169)\\
			(0,*),(0,*) & 1,3,3 & 125,24,24; 173 & \mathcal{B}_{175}(173)\\
			(0,0), (*,*) & 1,1,4 & 125,25,19;169 & \mathcal{B}_{175}(169)\sqcup \{(170,\mu)\mid\mu\in\mathcal{B}_5(4)\}^\circ\\ 
			(*,0),(*,*) & 1,2,4 & 125,20,19; 164 & \mathcal{B}_{175}(164) \sqcup \{(165,\mu)\mid\mu\in\mathcal{B}_{10}(9) \}^\circ\\ 
			(0,*),(*,*) & 1,3,4 & 125,24,19; 168 & \mathcal{B}_{175}(168)\sqcup \{(169,\mu)\mid\mu\in\mathcal{B}_6(5) \}^\circ\\
			(*,*),(*,*) & 1,4,4 & 125,19,19; 163 & \mathcal{B}'\ \text{(defined below)} \\
			\hline
			\end{array}\]
			\[ \mathcal{B}' := \mathcal{B}_{175}(163) \sqcup \{(164,\mu)\mid\mu\in\mathcal{B}_{11}(10) \}^\circ \sqcup \{(165,\nu)\mid\nu\in\mathcal{B}_{10}(8) \}^\circ. \]
		\end{small}
		\caption{Data on $\Omega(\phi)$ for $\phi=\phi(s_1,s_2,s_3)\in\Lin(P_{175})$ with $s_1=(0,0,0)$. This follows from Theorem~\ref{thm:10.6} when $\tau(s_i)\ne 4$ for all $i$, and Lemma~\ref{lem: omega star omega} otherwise.}\label{table:n=175}
	\end{table}
\end{eg}

\end{document}